\documentclass[11pt]{article}

\usepackage[T1]{fontenc}
\usepackage{lmodern}
\usepackage{soul}
\usepackage{empheq} 

\usepackage{amsmath, amssymb, amsthm}
\usepackage{mathrsfs}
\usepackage[margin=2.5cm]{geometry}
\usepackage{color, graphicx, float}
\usepackage[breakable]{tcolorbox}
\usepackage{tikz}
\usepackage[skip=3pt]{caption}
\usepackage{subcaption}
\usepackage{bbm}

\usepackage{etoolbox}
\makeatletter
\patchcmd{\@maketitle}{\LARGE \@title}{\LARGE\bfseries\@title}{}{}

\renewcommand{\@seccntformat}[1]{\csname the#1\endcsname.\quad}
\makeatother

\definecolor{darkblue}{rgb}{0,0,.5}

\usepackage{hyperref}
\hypersetup{
	colorlinks=true,		
	linkcolor=darkblue,		
	citecolor=darkblue,		
	urlcolor=darkblue 		
}

\makeatletter
\def\th@plain{%
	\thm@notefont{}
	\itshape 
}
\def\th@definition{%
	\thm@notefont{}
	\normalfont 
}

\renewenvironment{proof}[1][\proofname]{\par
	\normalfont
	\topsep0\p@\@plus3\p@ \trivlist
	\item[\hskip\labelsep\itshape
	#1\@addpunct{.}]\ignorespaces
}{%
	\qed\endtrivlist
}
\makeatother

\newtheorem{theorem}{Theorem}[section]
\newtheorem{lemma}[theorem]{Lemma}
\newtheorem{corollary}[theorem]{Corollary}
\newtheorem{proposition}[theorem]{Proposition}
\newtheorem{assumption}[theorem]{Assumption}
\theoremstyle{definition}
\newtheorem{definition}[theorem]{Definition}
\theoremstyle{definition}
\newtheorem{example}[theorem]{Example}
\theoremstyle{definition}
\newtheorem{remark}[theorem]{Remark}
\theoremstyle{definition}
\newtheorem{algorithm}{Algorithm}

\usepackage[shortlabels]{enumitem}

\allowdisplaybreaks

\newcommand{\dom}{\ensuremath{\operatorname{dom}}}
\newcommand{\Fix}{\ensuremath{\operatorname{Fix}}}
\newcommand{\zer}{\ensuremath{\operatorname{zer}}}
\newcommand{\gra}{\ensuremath{\operatorname{gra}}}
\newcommand{\Id}{\ensuremath{\operatorname{Id}}}

\newcommand{\ran}
{\ensuremath{\operatorname{ran}}}
\newcommand{\rank}{\ensuremath{\operatorname{rank}}}

\newcommand{\lspan}{\ensuremath{\operatorname{span}}}

\newcommand{\diag}{\ensuremath{\operatorname{diag}}}
\newcommand{\Deg}{\ensuremath{\operatorname{Deg}}}
\newcommand{\Inc}{\ensuremath{\operatorname{Inc}}}
\newcommand{\Lap}{\ensuremath{\operatorname{Lap}}}

\newcommand{\scal}[2]{\left\langle #1,#2 \right\rangle}

\def\bA{{\boldsymbol{A}}}
\def\bB{{\boldsymbol{B}}}
\def\bC{{\boldsymbol{C}}}
\def\bL{{\boldsymbol{L}}}
\def\bZ{{\boldsymbol{Z}}}
\def\ba{{\mathbf{a}}}
\def\bOne{{\mathbbm{1}}}
\def\bs{{\mathbf{s}}}

\def\bu{{\mathbf{u}}}
\def\bv{{\mathbf{v}}}
\def\bw{{\mathbf{w}}}
\def\bx{{\mathbf{x}}}
\def\by{{\mathbf{y}}}
\def\bz{{\mathbf{z}}}

\def\DB{{E}}

\begin{document}

\title{Primal-dual splitting for structured composite monotone inclusions with or without cocoercivity}

\author{
Minh N. Dao\thanks{School of Science, RMIT University, Melbourne, VIC 3000, Australia.
E-mail:~\href{href:minh.dao@rmit.edu.au}{minh.dao@rmit.edu.au}.},
~
Hung M.~Phan\thanks{Department of Mathematics and Statistics, Kennedy College of Sciences, University of Massachusetts Lowell, Lowell, MA 01854, USA.
E-mail:~\href{href:hung\char`_phan@uml.edu}{hung\char`_phan@uml.edu}.},
~
Matthew K. Tam\thanks{School of Mathematics and Statistics, The University of Melbourne, Parkville, VIC 3010, Australia.
E-mail:~\href{href:matthew.tam@unimelb.edu.au}{matthew.tam@unimelb.edu.au}.},
~and~
Thang D. Truong\thanks{School of Science, RMIT University, Melbourne, VIC 3000, Australia.
E-mail:~\href{href:thang.tdk64@gmail.com}{thang.tdk64@gmail.com}.}
}

\date{May 13, 2026}

\maketitle

\begin{abstract}
In this paper, we propose a primal-dual splitting algorithm for a broad class of structured composite monotone inclusions that involve finitely many set-valued operators, compositions of set-valued operators with bounded linear operators, and single-valued operators possibly without cocoercivity. The proposed algorithm is not only a unification for several contemporary algorithms but also a blueprint to generate new algorithms with graph-based structures using a single transparent convergence analysis. Our approach reduces dimensionality compared with the standard product space technique, which typically reformulates the original problem as the sum of two maximally monotone operators in order to apply splitting methods. It accommodates different cocoercive or Lipschitz constants as well as different resolvent parameters, and yields a larger allowable stepsize range than recent methods. We demonstrate the practicality of the approach by a numerical experiment on cancer detection using the decentralized fused LASSO problem.  
\end{abstract}


\noindent{\bf Mathematics Subject Classification (MSC 2020):}
47H05,  
47H10,  
65K10,  
90C30.  

\section{Introduction}

Many interesting problems such as composite optimization problems, structured saddle-point problems, and variational inequalities can be formulated as monotone inclusion problems \cite{ComP11, Roc70, RW98}. Among the most widely used methods for solving such problems are \emph{splitting algorithms}, in which computations are performed separately on each operator instead of their sums, see, e.g., \cite{BC17, DR56, LM79, MT23, Ryu20, Tseng00}. When an optimization problem involves compositions of set-valued operators with bounded linear operators, common mathematical tools such as resolvents can be computed for such compositions, however, it is generally very difficult, thus, not desirable. Instead, \emph{primal-dual splitting algorithms} \cite{ABT23, BAC11, ChamP11, Con13, Vu13} are usually used to process the set-valued operators separately using backward steps via their resolvents, while the bounded linear operators are evaluated directly using forward steps on their own or on their adjoints. Considering the primal inclusion and the dual inclusion simultaneously provides us a better understanding of the problem.

Let $\mathcal{H}$ and $(\mathcal{G}_k)_{1\leq k\leq r}$ be real Hilbert spaces. We consider the primal inclusion
\begin{align}
\label{primal_prob}
    \text{find } x \in \mathcal{H} \text{ such that } 0 \in \sum_{i=1}^n A_i x + \sum_{k=1}^r L_k^* B_k L_k x + \sum_{j=1}^p C_jx,
\end{align}
where, for each $i\in \{1, \dots, n\}$, $A_i\colon \mathcal{H} \rightrightarrows \mathcal{H}$ is a maximally monotone operator, for each $k\in \{1, \dots, r\}$, $B_k\colon \mathcal{G}_k \rightrightarrows \mathcal{G}_k$ is a maximally monotone operator and $L_k\colon \mathcal{H}\rightarrow \mathcal{G}_k$ is a bounded linear operator with adjoint operator $L_k^*$, and for each $j\in \{1, \dots, p\}$, $C_j\colon \mathcal{H}\to \mathcal{H}$ is either a cocoercive operator, or a monotone and Lipschitz continuous operator. Problem \eqref{primal_prob} arises in a wide range of applications, including location problems \cite{BH13}, image reconstruction \cite{ChamP11}, and signal processing \cite{ComP11}. The associated dual inclusion in the sense of Attouch--Th{\'e}ra \cite{AT96} is
\begin{align}
\label{dual_prob}
\begin{aligned}
    &\text{find } (s_1,\dots,s_r) \in \prod_{k=1}^r\mathcal{G}_k \text{~such that~} \\
    &0 \in - L_k\left(\sum_{i=1}^n A_i + \sum_{j=1}^p C_j\right)^{-1} \left(-\sum_{k=1}^r L_k^*s_k\right) + B_k^{-1} s_k,\quad k\in\{1,\dots,r\}.
\end{aligned}
\end{align}

When $n\neq 0$, $r=0$, and $p=0$, problem~\eqref{primal_prob} reduces to finding a zero in the sum of finitely many maximally monotone operators that are potentially set-valued. The most well-known algorithm for the classical case $n=2$ is the \emph{Douglas--Rachford algorithm} \cite{DR56}. Recently, a splitting algorithm was proposed for $n=3$ by Ryu \cite{Ryu20}, while a different resolvent splitting for the general case $n\geq 2$ was introduced by Malitsky and Tam \cite{MT23}. The latter work was then extended into a framework \cite{Tam23} that covers these algorithms. In the setting when $n \neq 0$, $r=0$, and $p \neq 0$, the problems involve not only maximally monotone operators but also single-valued operator which can be used directly by forward evaluations. When $n=1$ and $p=1$, the \emph{forward-backward algorithm} \cite{LM79} is typically used when $C$ is cocoercive, while the \emph{forward-backward-forward algorithm}~\cite{Tseng00} and the \emph{forward-reflected-backward algorithm} \cite{MT20} are commonly applied when $C$ is monotone and Lipschitz continuous. Moreover, the authors of \cite{AMTT23} developed a \emph{distributed forward-backward algorithm} for the case where $n\geq 2$, $p=n-1$, and each $C_j$ is cocoercive, as well as a second algorithm for $n\geq 3$, $p=n-2$, and each $C_j$ monotone and Lipschitz continuous. Recently, the \emph{forward-backward algorithms devised by graphs} in \cite{ACL24} and a class of algorithms in \cite{ACGN25} effectively characterized methods that use only individual resolvent evaluations and direct evaluations of cocoercive operators, while a general approach to distributed operator splitting \cite{DTT26} addresses the situations where the single-valued operators may not be cocoercive.

In the setting when $r\neq 0$, there are many primal-dual algorithms are proposed, but only for some special cases of problem~\eqref{primal_prob}. For example, in the case $n=1$, $r=1$, and $p=0$, popular algorithms include the \emph{Chambolle--Pock algorithm} \cite{ChamP11} and an algorithm proposed by Brice\~no-Arias and Combettes \cite{BAC11}. When $n=0$, $r=1$, and $p=1$, one can use \emph{primal-dual fixed-point algorithm based on the proximity operator} (PDFP$^2$O) or \emph{proximal alternating predictor-corrector} (PAPC) \cite{CHZ13, DST15}. For the case where $n=1$, $r=1$, and $p=1$, or when a Lipschitz and a cocoercive operator are treated simultaneously, we refer interested readers to \cite{ComP12, Con13, Rol24, Vu13}. When $r \geq 2$, one can adapt some of the above algorithms and use product space reformulations, even though it may not utilize different structures of algorithm design. Recently, Arag\'on-Artacho~et~al. \cite{ABT23} studied the case with $n\geq 2$, $r=1$, and $p=0$ in the reduced dimension $n-1$. Then, to incorporate multiple bounded linear operators with $r \geq 2$, the proposed algorithm, however, still relies on a product space reformulation. 

Most of the aforementioned algorithms reformulate the original problem as the sum of a maximally monotone operator and a skew-symmetric or cocoercive operator, then apply the forward-backward \cite{Con13, Vu13} or forward-backward-forward approach \cite{BAC11, ComP12}. In this work, we introduce a primal-dual splitting algorithm for structured composite monotone inclusions that unifies many well-known algorithms in the literature through a single transparent convergence analysis. Our algorithm provides a different perspective using reduced dimension $n-1+r$ compared to dimension $n+r$ using product space reformulations which potentially contribute to the theory and applications of primal-dual algorithms. Furthermore, we directly handle the linear operators $L_1, \dots, L_r$, rather than through a product space reformulation, which allows us to generate alternative designs with potentially significantly larger stepsizes compared to those in \cite[Corollary 1]{ABT23}. Therefore, this approach facilitates distributed computation without central coordination which helps prevent bottlenecks that arise from centralized coordination in practice. In addition, different cocoercive or Lipschitz constants, as well as resolvent parameters are also allowed for distributed implementation without requiring the knowledge of a global cocoercive or Lipschitz constant.

In this work, our main contributions are as follows.
\begin{enumerate}
    \item 
    We develop a primal-dual algorithm for structured composite monotone inclusions involving finitely many set-valued operators, composition of set-valued operators with bounded linear operators, and single-valued operators that may not be cocoercive. Our algorithm does not rely on reformulating the original problem as the sum of a maximally monotone operator and a skew-symmetric or cocoercive operator using product space reformulations but rather directly exploit the problem structure to derive different algorithm designs, offer possibly larger stepsize ranges with explicit formulas. For our convergence analysis, we introduce the concept of \emph{quasicomonotonicity} in Definition~\ref{def:}\ref{def:quasi_como} and its natural connection to quasiaveragedness \cite[Definition 2.1]{DTT26}.
    \item 
    The proposed algorithm not only unifies several contemporary methods but also provides a general framework for constructing new graph-based algorithms within a single unified convergence analysis. Notably, our analysis enables the derivation of broader admissible stepsize ranges under weaker assumptions on the coefficient matrices compared with recent methods \cite{ACGN25, ABT23, DTT26}. Moreover, the algorithm also allows different cocoercive or Lipschitz constants as well as different resolvent parameters suitable for distributed implementation, thereby improving upon certain existing methods. Finally, we explore various choices of the coefficient matrices of the algorithm and examine the effect of stepsizes on the algorithm's performance in a numerical experiment on cancer detection using the decentralized fused LASSO problem.
\end{enumerate}

The remainder of this paper is structured as follows. In Section~\ref{s:prelim}, we introduce some notations and background materials on set-valued, single-valued operators, Kronecker product, and solution sets of the primal and dual problem. In Section~\ref{s:framework}, we present our primal-dual splitting algorithm in Algorithm~\ref{algo:full} with main convergence results in Theorem~\ref{t:cvg}. We discuss the relation to existing algorithms and new algorithms in Section~\ref{s:realizations}. Section~\ref{s:num_exper} provides a numerical experiment on the decentralized fused LASSO problem for cancer detection, examining the performance of the algorithm under different settings of the coefficient matrices and parameters.

\section{Preliminaries}
\label{s:prelim}

In this paper, the sets of nonnegative integers and real numbers are denoted by $\mathbb{N}$ and $\mathbb{R}$, respectively. We assume that $\mathcal{H}$ and $(\mathcal{G}_k)_{1\leq k\leq r}$ are real Hilbert spaces equipped with their respective inner product $\scal{\cdot}{\cdot}$ and induced norm $\|\cdot\|$. Unless otherwise stated, we use the standard inner product in the product space $\mathcal{H}^n$, i.e., for $\bx=(x_1,\ldots,x_n)\in \mathcal{H}^n$ and $\bar\bx=(\bar{x}_1,\ldots,\bar{x}_n) \in\mathcal{H}^n$,
\begin{align*}
\scal{\bx}{\bar\bx} =\sum_{i=1}^n\scal{x_i}{\bar{x}_i}.
\end{align*}
With this inner product, we say that a linear operator $\Psi:\mathcal{H}^n\to\mathcal{H}^n$ is {\em positive semidefinite}, denoted by $\Psi\succeq 0$ if, for all $\bx\in\mathcal{H}^n$, $\scal{\bx}{\Psi\bx}\geq0$. Strong and weak convergence of sequences are denoted by $\rightarrow$ and $\rightharpoonup$, respectively. 

For an operator $A$ on $\mathcal{H}$, we write $A\colon \mathcal{H}\rightrightarrows \mathcal{H}$ when $A$ is set-valued, and $A\colon \mathcal{H}\to \mathcal{H}$ when $A$ is single-valued. For such an operator $A$, the \emph{domain}, \emph{graph}, \emph{fixed-point set}, and \emph{zero set} are defined by
$\dom A :=\{x\in \mathcal{H}: Ax\neq \varnothing\}$, $\gra A :=\{(x,u)\in \mathcal{H}\times \mathcal{H}: u\in Ax\}$, $\Fix A :=\{x\in \mathcal{H}: x\in Ax\}$, and $\zer A:=\{x\in\mathcal{H}: 0 \in Ax\}$, respectively. The \emph{resolvent} of $A$ is defined by $J_A :=(\Id +A)^{-1}$, where $\Id$ denotes the identity operator. We say that $A$ is \emph{monotone} if, for all $(x, u), (y, v)\in \gra A$,
\begin{align*}
\scal{x -y}{u -v} \geq 0
\end{align*}
and \emph{maximally monotone} if it is monotone and there exists no monotone operator whose graph properly contains $\gra A$.

An operator $C\colon \mathcal{H}\to \mathcal{H}$ is \emph{$\ell$-Lipschitz continuous} for $\ell\in [0, +\infty)$ if, for all $x, y\in \mathcal{H}$, $\|Cx -Cy\| \leq \ell\|x -y\|$
and \emph{$\frac{1}{\ell}$-cocoercive} for $\ell\in (0, +\infty)$ if, for all $x, y\in \mathcal{H}$, $\scal{Cx -Cy}{x -y} \geq \frac{1}{\ell}\|Cx -Cy\|^2$.
By the Cauchy--Schwarz inequality, every $\frac{1}{\ell}$-cocoercive operator is monotone and $\ell$-Lipschitz continuous, which, in turn, means that it is also maximally monotone \cite[Corollary 20.28]{BC17}. 

For our convergence analysis, we will recall several related concepts.

\begin{definition}
\label{def:}
We say that
\begin{enumerate}
\item $\Gamma:\mathcal{H}\to\mathcal{H}$ is \emph{$\alpha$-comonotone} \cite[Definition 2.4]{BMW20} with $\alpha \in \mathbb{R}$ if, for all $x,y\in \mathcal{H}$,
\begin{align*}
    \scal{\Gamma x - \Gamma y}{x -y} \geq \alpha\|\Gamma x - \Gamma y\|^2.
\end{align*}
\item \label{def:quasi_como}
$\Gamma$ is {\em $\alpha$-quasicomonotone} (with respect to the set $\zer\Gamma$) if, for all $x\in\mathcal{H}$ and $y\in\zer \Gamma $,
\begin{align*}
    \scal{\Gamma x}{x -y} \geq \alpha\| \Gamma x\|^2.
\end{align*}
Unless stated otherwise, we will simply say $\Gamma$ is $\alpha$-quasicomonotone and drop the reference set $\zer\Gamma$.
\item $T\colon \mathcal{H}\to \mathcal{H}$ is \emph{conically $\rho$-averaged} \cite[Definition~2.1]{BDP22} with $\rho \in (0, +\infty)$ if, for all $x, y\in \mathcal{H}$, 
\begin{align*}
\label{averaged}
\|Tx -Ty\|^2 +\frac{1-\rho}{\rho}\|(\Id -T)x -(\Id -T)y\|^2 \leq \|x -y\|^2,
\end{align*}
\item $T$ is \emph{conically $\rho$-quasiaveraged} (with respect to the set $\Fix T$) \cite[Definition 2.1]{DTT26} if, $\rho\in (0, +\infty)$ and for all $x\in \mathcal{H}$ and all $y\in \Fix T$, 
\begin{align*}
\|Tx -y\|^2 +\frac{1-\rho}{\rho}\|(\Id -T)x\|^2 \leq \|x -y\|^2.
\end{align*}
Unless stated otherwise, we will simply say $T$ is conically $\rho$-quasiaveraged and drop the reference set $\Fix T$.
\end{enumerate}
\end{definition}
Clearly, every conically $\rho$-averaged operator is conically $\rho$-quasiaveraged.
It is known from \cite[Proposition~3.3]{BDP22} that, for $\alpha, \theta\in (0, +\infty)$, $\Gamma$ is $\alpha$-comonotone if and only if $\Id-\theta\Gamma$ is conically $\frac{\theta}{2\alpha}$-averaged. Naturally, we will derive a similar connection between conical quasiaveragedness and quasicomonotonicity.

\begin{proposition}
\label{p:quasi_co_averaged}
    Let $\Gamma \colon \mathcal{H} \to \mathcal{H}$ and $\alpha,\theta \in (0, +\infty)$. Then $\Gamma$ is $\alpha$-quasicomonotone if and only if $T:=\Id-\theta \Gamma$ is conically $\frac{\theta}{2\alpha}$-quasiaveraged.
\end{proposition}
\begin{proof}
Take $\bar{x}\in\Fix T=\zer\Gamma$ and take any $x\in\mathcal{H}$.
Then $T$ is conically $\frac{\theta}{2\alpha}$-quasiaveraged
\begin{alignat*}{2}
&\iff\quad
&\|Tx-\bar{x}\|^2 + \frac{1-\theta/(2\alpha)}{\theta/(2\alpha)}\|(\Id-T)x\|^2 &\leq \|x-\bar{x}\|^2\\
&\iff  &
\|x-\theta\Gamma x - \bar{x}\|^2
+\frac{2\alpha-\theta}{\theta}
\|\theta\Gamma x\|^2 &\leq \|x-\bar{x}\|^2\\
&\iff &
-2\scal{\theta\Gamma x}{x-\bar{x}}
+\|\theta\Gamma x\|^2
+(2\alpha-\theta)\theta\|\Gamma x\|^2&\leq 0\\
&\iff & \alpha\|\Gamma x\|^2 &\leq \scal{\Gamma x}{x-\bar{x}},
\end{alignat*}
which means $\Gamma$ is $\alpha$-quasicomonotone.
\end{proof}

\subsection{Kronecker product, the vector \texorpdfstring{$\bOne$}{1}, and the diagonal set \texorpdfstring{$\Sigma$}{Sigma}}

For a matrix $M$, we denote its \emph{range} and \emph{kernel} by $\ran M$ and $\ker M$, respectively, its \emph{transpose} by $M^\top $, and its \emph{pseudo-inverse} by $M^\dag$. For a positive diagonal matrix $\DB=\diag(\eta_1,\ldots,\eta_n)$ where $\eta_i \in(0,+\infty)$, we denote its square root matrix by $\sqrt{\DB}=\diag(\sqrt{\eta_1},\ldots,\sqrt{\eta_n})$. Clearly, $\sqrt{\DB}\sqrt{\DB}=\DB$.

In an abuse of notation, we identify matrices $M=(M_{ij})\in\mathbb{R}^{n\times m}$ with the Kronecker product $(M \otimes \Id) \colon \mathcal{H}^m \to \mathcal{H}^n$. Therefore, for any $\bz =(z_1,\dots,z_m)\in\mathcal{H}^m$, we have
\begin{equation*}
M\bz 
:=(M\otimes \Id)\bz 
=\left(\sum_{j=1}^{m}M_{1j}z_j,
\sum_{j=1}^{m}M_{2j}z_j,\ldots,
\sum_{j=1}^{m}M_{nj}z_j\right)\in\mathcal{H}^n.
\end{equation*}
Let $A_1, \dots, A_n \colon\mathcal{H} \rightrightarrows \mathcal{H}$ be maximally monotone operators. Given $\bx=(x_1, \dots, x_n)\in \mathcal{H}^n$, we define the operator $\bA\colon \mathcal{H}^n \rightrightarrows \mathcal{H}^n$ by $\bA\bx := A_1 x_1\times \dots\times A_n x_n$. It follows that $\bA $ is also a maximally monotone operator. As a consequence, its resolvent $J_{\bA }\colon \mathcal{H}^n \rightarrow \mathcal{H}^n$ is given by $J_{\bA }=(J_{A_1}, \dots, J_{A_n})$.
Note that we can write $\bA $ and $J_{\bA }$ as diagonal operators
\begin{align*}
\bA =\begin{bmatrix}
A_1 & & 0\\
 & \ddots & \\
0 & & A_n
\end{bmatrix}\text{ and }
J_\bA =\begin{bmatrix}
J_{A_1} & & 0\\
    & \ddots & \\
0   & & J_{A_n}
\end{bmatrix}.
\end{align*}
Similarly, we define the diagonal operators
\begin{align*}
\bB :=\begin{bmatrix}
B_1 &  & 0\\
    & \ddots & \\
0   &    & B_r
\end{bmatrix},\ 
\bL :=\begin{bmatrix}
L_1 &  & 0\\
    &  \ddots& \\
0   &  & L_r
\end{bmatrix},\ 
\bC :=\begin{bmatrix}
C_1 &  & 0\\
    & \ddots & \\
0   & & C_p
\end{bmatrix},
\end{align*}
where the operators $B_k$, $L_k$, and $C_j$ are from the setup of problem~\eqref{primal_prob}.

We denote $\bOne_n=(1,\ldots,1)\in\mathbb{R}^n$. When the context is clear, we will drop the subscript and simply write $\bOne$, for example, if $M\in\mathbb{R}^{m\times n}$, then $\bOne^\top M\bOne=(\bOne_m)^\top M(\bOne_n)$. 
Using this notation, we can define the \emph{diagonal set}
\begin{align*}
\Sigma_n := \{ \bx =(x, \dots, x) \in \mathcal{H}^n\}
=\{(\bOne_n) x\ :\ x\in\mathcal{H}\},
\end{align*}
and thus its the orthogonal set is
$\Sigma_n^\bot=\{\hat{\bx} \in\mathcal{H}^n \,:\, \bOne^\top \hat{\bx} = 0\}$. Again, we will drop the subscript and denote $\bOne x$ when the dimension is clear. In particular, for $x\in\mathcal{H}$,
$\bA(\bOne x) =(A_1 x,\ldots, A_n x)$,
$\bL(\bOne x) =(L_1 x,\ldots, L_r x)$, and 
$\bC(\bOne x) =(C_1 x,\ldots, C_p x)$.

\subsection{Solution sets of the primal and dual problems}

The sets of solutions to \eqref{primal_prob} and \eqref{dual_prob} are denoted by $\mathcal{P}$ and $\mathcal{D}$, respectively. Let $\bZ $ be a subset of $\mathcal{P} \times \mathcal{D}$ such that
\begin{align}\label{e:Z_set}
    \bZ  &= \Bigg\{(x,\bs )\in\mathcal{H}\times \prod_{k=1}^r\mathcal{G}_k : \ \bs =(s_1,\dots,s_r)\in \prod_{k=1}^r B_k L_k x \text{~and~} 0 \in \sum_{i=1}^n A_i x + \sum_{k=1}^r L_k^*s_k + \sum_{j=1}^p C_j x  
    \Bigg\} \notag\\
    &=\bigg\{(x,\bs )\in \mathcal{H}\times \prod_{k=1}^r\mathcal{G}_k :\ \bs\in \bB\bL(\bOne x)
    \text{~~and~~}
    0\in \bOne^\top \bA(\bOne x) + \bOne^\top \bL^*\bs  + \bOne^\top \bC(\bOne x) \bigg\}.
\end{align}
An element of $\bZ$ is referred to as a \emph{primal-dual solution} to \eqref{primal_prob} and \eqref{dual_prob} due to the following statement.
\begin{equation}
    \label{e:sol_sets}
    \mathcal{P} \neq \varnothing \iff \bZ \neq \varnothing \iff \mathcal{D} \neq \varnothing.
\end{equation}
To prove \eqref{e:sol_sets}, we have
\begin{align*}
    \exists \ x \in \mathcal{P}
    &\iff \exists \ x\in\mathcal{H}, \quad 0 \in \sum_{i=1}^n A_i x + \sum_{k=1}^r L_k^* B_k L_k x + \sum_{j=1}^p C_jx\\
    &\iff \exists \ x\in\mathcal{H}, \quad 0\in \bOne^\top \bA(\bOne x) + \bOne^\top \bL^*\bB\bL(\bOne x) + \bOne^\top \bC(\bOne x)  \\
    &\iff \exists \ (x,\bs)\in \mathcal{H}\times \prod_{k=1}^r\mathcal{G}_k, \quad 
    \begin{cases}
        -\bOne^\top \bL^* \bs &\in \bOne^\top \bA(\bOne x) + \bOne^\top \bC(\bOne x),\\
        \bs &\in \bB\bL(\bOne x)
    \end{cases}\\
    &\iff \exists \ (x,\bs)\in \mathcal{H}\times \prod_{k=1}^r\mathcal{G}_k, \quad
    \begin{cases}
        \bOne x &\in \Big(\bOne^\top \bA(\bOne x) + \bOne^\top \bC(\bOne x)\Big)^{-1} \Big(-\bOne^\top \bL^* \bs\Big)\\
        \bL(\bOne x) &\in \bB^{-1}\bs
    \end{cases}\\
    &\iff \exists \ \bs \in \prod_{k=1}^r\mathcal{G}_k, \quad 0 \in - \bL\Big(\bOne^\top \bA(\bOne x) + \bOne^\top \bC(\bOne x)\Big)^{-1} \Big(-\bOne^\top \bL^* \bs\Big) + \bB^{-1}\bs \\
    &\iff \exists \ \bs \in \prod_{k=1}^r\mathcal{G}_k, \quad 0 \in - L_k\Big(\sum_{i=1}^n A_i + \sum_{j=1}^p C_j\Big)^{-1} \Big(-\sum_{k=1}^r L_k^*s_k\Big) + B_k^{-1} s_k,\ \ k\in\{1,\dots,r\} \\
    &\iff \exists \ \bs \in \mathcal{D}.
\end{align*}

\section{A primal-dual splitting algorithm}
\label{s:framework}

Throughout this section, for each $i\in \{1, \dots, n\}$, each $k\in \{1, \dots, r\}$, and each $j\in \{1, \dots, p\}$, $A_i$ and $B_k$ are maximally monotone operators, $L_k$ is a bounded linear operator with adjoint operator $L_k^*$, and $C_j$ is a monotone and $\ell_j$-Lipschitz continuous operator. 
Let $P\in\mathbb{R}^{n\times p}$, $Q\in\mathbb{R}^{n\times p}$, and $R\in \mathbb{R}^{p\times n}$. Define an auxiliary operator $\Phi:\mathcal{H}^n\to\mathcal{H}^n$ by
\begin{align}\label{e:Phi}
\Phi:= (P-Q)\bC R + Q\bC P^\top.
\end{align}
Let $D =\diag(\delta_1, \dots, \delta_n)\in\mathbb{R}^{n\times n}$ and $\DB = \diag(\eta_1,\dots, \eta_r)\in\mathbb{R}^{r\times r}$ be positive diagonal matrices. Let $M \in \mathbb{R}^{n \times m}$, $N\in \mathbb{R}^{n \times n}$, $H\in\mathbb{R}^{n\times r}$, and $K\in \mathbb{R}^{r\times n}$. Given $\bz\in\mathcal{H}^{m}$ and $\bw\in\prod_{k=1}^r \mathcal{G}_k$, we define the operator $T\colon \mathcal{H}^{m}\times \prod_{k=1}^r \mathcal{G}_k \to \mathcal{H}^{m}\times \prod_{k=1}^r \mathcal{G}_k$ as
\begin{tcolorbox}[boxsep=0pt, left=0pt, right=0pt,
top=-10pt, bottom=0pt,
colback=blue!10!white,colframe=blue!30!white,
boxrule=0pt,breakable, after={\parindent=0pt}]
\begin{align}
\label{e:fix_oper_T}
T(\bz, \bw)
:= \left[
    \begin{array}{c}
        T_1(\bz ,\bw )\\
        T_2(\bz ,\bw )
    \end{array}\right]
    :=\left[
    \begin{array}{c}
        \bz \\
        \bw 
    \end{array}\right] -
    \left[
    \begin{array}{c}
        M^\top\bx\\
        \DB(\bL H^\top\bx - \by)
    \end{array}\right],
\end{align}
\end{tcolorbox}
where $\bx, \by$ are defined via the solution operator $S\colon \mathcal{H}^{m}\times \prod_{k=1}^r \mathcal{G}_k \to \mathcal{H}^{n}\times \prod_{k=1}^r \mathcal{G}_k$ with $\gamma \in (0, +\infty)$
\begin{tcolorbox}[boxsep=0pt, left=0pt, right=0pt, 
top=-10pt, bottom=0pt,
colback=blue!10!white, colframe=blue!30!white,
boxrule=0pt, breakable]
\begin{align}
\label{e:sol_oper_S}
(\bx, \by) := S(\bz, \bw) 
    :=
    \left[
    \begin{aligned}
        &J_{\gamma D^{-1}\bA }\left( D^{-1} \left(M\bz  + N\bx  - \gamma 
        \Phi\bx -\gamma H\bL^*(\DB\bL K\bx -\bw )\right)\right)\\
        &J_{\DB^{-1}\bB }\Big(\bL K\bx -\DB^{-1}\bw  + \bL H^\top \bx \Big)
    \end{aligned}\right].
\end{align}
\end{tcolorbox}

We propose the following primal-dual algorithm.

\begin{tcolorbox}[boxsep=0pt,left=0pt,right=0pt,top=0pt,bottom=0pt, colback=blue!10!white,colframe=blue!30!white,
    boxrule=0pt,breakable]
\begin{algorithm}
\label{algo:full}
Let $\mathbf{z}^0\in\mathcal{H}^m$, $\bw^0\in\prod_{k=1}^r \mathcal{G}_k$, and $(\lambda_t)_{t\in\mathbb{N}} \subset (0,+\infty)$. For each $t\in\mathbb{N}$, compute
\begin{align*}
    (\bz^{t+1}, \bw^{t+1}) := \Big(1-\lambda_t\Big)(\bz^t, \bw^t) + \lambda_t T(\bz^t, \bw^t), 
\end{align*}
where $\bx^t, \by^t$ in $T(\bz^t, \bw^t)$ are determined by 
\begin{align*}
    (\bx^t, \by^t) := S(\bz^t, \bw^t)
\end{align*}
with $T$, $S$ defined by \eqref{e:fix_oper_T} and \eqref{e:sol_oper_S}, respectively.
\end{algorithm}
\end{tcolorbox}

\begin{remark}[Conditions for explicitness]
\label{r:OnAlgo}
Algorithm~\ref{algo:full} can be written in the form
\begin{align*}
\begin{aligned}
&\begin{cases}
x^t_i &=J_{\frac{\gamma}{\delta_i}A_i}\Bigg(\frac{1}{\delta_i}\sum_{j=1}^m M_{ij}z^t_j +\frac{1}{\delta_i}\sum_{j=1}^{n} N_{ij}x^t_j -\frac{\gamma}{\delta_i}\sum_{j=1}^p (P_{ij} -Q_{ij})C_j\Big(\sum_{k=1}^{n} R_{jk}x^t_k\Big) \\ 
&\qquad -\frac{\gamma}{\delta_i}\sum_{j=1}^p Q_{ij}C_j\Big(\sum_{k=1}^{n} P_{kj}x^t_k\Big) - \frac{\gamma}{\delta_i} \sum_{j=1}^r H_{ij}L_j^*\Big(\eta_j L_j\Big(\sum_{k=1}^n K_{jk} x_k^t \Big) - w_j^t\Big)\Bigg), \\
&\qquad \quad i \in \{1, \dots, n\}\\
y_k^t &= J_{\frac{1}{\eta_k}B_k}\Bigg(L_k\Big(\sum_{j=1}^{n} K_{kj}x_j^t\Big) -  \frac{1}{\eta_k} w_k^t + L_k\Big(\sum_{j=1}^n H_{jk} x_j^t\Big)\Bigg), \quad k\in\{1,\dots,r\},
\end{cases}\\
&\begin{cases}
    z^{t+1}_i &= z^t_i -\lambda_t\sum_{j=1}^n M_{ji}x^t_j,\quad i\in \{1, \dots, m\} \\
    w_k^{t+1} &= w_k^t - \lambda_t \eta_k \Big(L_k(\sum_{j=1}^{n} H_{jk} x_j^t) - y_k^t\Big), \quad k\in\{1,\dots, r\}.
\end{cases}
\end{aligned}
\end{align*}
In general, the algorithm is implicit in the sense that the calculation of $\bx^t$ requires the value of $\bx^t$ itself. When the first rows of $N$, $P$, $Q$, and $H$ are zeros, $x_1^k$ is explicitly expressed as 
\begin{align}
\label{eq:explicit_x1}
    x^t_1 &=J_{\frac{\gamma}{\delta_1}A_1}\left(\frac{1}{\delta_1}\sum_{j=1}^m M_{1j}z^t_j\right),
\end{align}
but the remaining components $x_i^t$ may still be implicit. 
Although our convergence analysis of Algorithm~\ref{algo:full} holds in this implicit setting, for practical implementation, we are interested in explicit versions, in which $x_1^t$ depends only on $\bz^t$ and, for each $i\in \{2,\dots, n\}$, the update of $x_i^t$ depends only on $\bz^t$ and components $x_j^t$ with $j<i$ that have already been computed. This property is guaranteed as soon as $N$, $\Phi$, and $H\bL^*\DB\bL K$ are strictly lower triangular, equivalently, for all $k\geq i$, 
\begin{subequations}
\label{a:PQR_lower}
\begin{empheq}[left=\empheqlbrace]{align}
    N_{ij} &= 0, \quad i \le j 
        \label{a:PQR_lower_1a} \\
    (P_{ij} - Q_{ij}) R_{jk} &= 0, \quad j\in\{1,\dots,p\}
        \label{a:PQR_lower_1b} \\
    Q_{ij} P_{kj} &= 0, \quad j\in\{1,\dots,p\}
        \label{a:PQR_lower_1c} \\
    H_{ij} K_{jk} &= 0, \quad j\in\{1,\dots,r\}.
        \label{a:PQR_lower_1d}
\end{empheq}
\end{subequations}
Equations \eqref{a:PQR_lower_1b} and \eqref{a:PQR_lower_1c} mean that, for each $(i,j)\in\{1,\dots,n\}\times\{1,\dots,p\}$, exactly one of the following happens:
\begin{itemize}
    \item $R_{ji} \neq 0,\ P_{1j}=\dots=P_{ij}=0,\ Q_{1j}=\dots = Q_{ij}=0$;
    \item $R_{ji}=0,\ P_{ij} \neq 0,\  Q_{1j}=\dots=Q_{ij}=0$;
    \item $R_{ji}=0,\ P_{ij}=0, \ Q_{ij} \neq 0$.
\end{itemize}
These conditions, in fact, can be visualized nicely using a block structure in which the matrices $P$, $Q$, and $R^\top$ have complementary structure as in Figure~\ref{fig:PQR}. The possible nonzero parts of $R$, $P$, and $Q$ are in blue, green, and red colors, respectively. Similarly, the conditions on $H$ and $K$ are illustrated in Figure~\ref{fig:HK}.
\begin{figure}[htbp]
    \centering
    \begin{subfigure}{0.4\textwidth}
    \centering
    
    \tikzset{every picture/.style={line width=1.5pt}}  

    \begin{tikzpicture}[x=0.75pt,y=0.75pt,yscale=-1,xscale=1]
    
    \draw    (89.5,40.5) -- (119.58,40.5) ;
    \draw    (89.5,40.5) -- (91,265) ;
    \draw    (91,265) -- (116.98,265) ;
    \draw    (311,41.8) -- (283.65,41.8) ;
    \draw    (311,41.8) -- (311,267) ;
    \draw    (311,267) -- (283.65,267) ;
    \draw    (90.87,91.04) -- (153,91) ;
    \draw    (153,91) -- (153.04,119.66) ;
    \draw    (153.04,119.66) -- (235,120) ;
    \draw    (91,182) -- (146,182) ;
    \draw    (146,182) -- (146,207) ;
    \draw    (146,207) -- (203.43,207) ;
    \draw    (203,173) -- (203.43,207) ;
    \draw    (203,173) -- (311.02,173) ;
    \draw    (235,120) -- (235,102) ;
    \draw    (235,102) -- (311,102) ;
    
    \fill[blue!25, draw=none]
      (89.5,40.5) --
      (119.58,40.5) --
      (311,41.8) --
      (311,102) --
      (235,102) --
      (235,120) --
      (153.04,119.66) --
      (153,91) --
      (90.87,91.04) --
      cycle;
    
    \fill[green!25, draw=none]
      (90.87,91.04) --
      (153,91) --
      (153.04,119.66) --
      (235,120) --
      (235,102) --
      (311,102) --
      (311.02,173) --
      (203,173) --
      (203.43,207) --
      (146,207) --
      (146,182) --
      (91,182) --
      cycle;
    
    \fill[red!25, draw=none]
      (91,182) --
      (146,182) --
      (146,207) --
      (203.43,207) --
      (203,173) --
      (311.02,173) --
      (311,267) --
      (283.65,267) --
      (91,265) --
      cycle;
    
    \draw (98,50) node [anchor=north west][inner sep=0.75pt]  [font=\normalsize] [align=left]{\parbox{5cm}{$R_{ji} \neq 0\ (P_{1j}=\dots=P_{ij} =0$;\\
    \hspace*{1.3cm}  $Q_{1j} =\dots =Q_{ij} \ =0)$}};
    \draw (117.71,127.44) node [anchor=north west][inner sep=0.75pt]  [font=\normalsize] [align=left] {\parbox{4cm}{$P_{ij} \neq 0\ ( R_{ji} =0;\\
    \hspace*{0.3cm}Q_{1j} =\dotsc =Q_{ij} =0)$}};
    \draw (108.79,229) node [anchor=north west][inner sep=0.75pt]  [font=\normalsize] [align=left] {\parbox{5cm}{$Q_{ij} \neq 0
    \ ( R_{ji} =0;\ P_{ij} =0)$}};
    \end{tikzpicture}

    \caption{$P, Q, R$ structure.}
    \label{fig:PQR}
    \end{subfigure}
    \hfill
    \begin{subfigure}{0.4\textwidth}
    \centering
    \tikzset{every picture/.style={line width=1.5pt}}  

    \begin{tikzpicture}[x=0.75pt,y=0.75pt,yscale=-1,xscale=1]
    
    \draw    (97.87,112.59) -- (146,113) ;
    \draw    (146,113) -- (146,189.54) ;
    \draw    (173,129) -- (244,129) ;
    \draw    (97.5,34.5) -- (127.58,34.5) ;
    \draw    (97.5,34.5) -- (99,259) ;
    \draw    (99,259) -- (124.98,259) ;
    \draw    (319,35.8) -- (291.65,35.8) ;
    \draw    (319,35.8) -- (319,162) -- (319,261) ;
    \draw    (319,261) -- (291.65,261) ;
    \draw    (146,189.54) -- (173,190) ;
    \draw    (173,190) -- (173,129) ;
    \draw    (244,129) -- (244,170) ;
    \draw    (244,170) -- (280,170) ;
    \draw    (280,149) -- (280,170) ;
    \draw    (280,149) -- (319,149) ;

    \fill[blue!25, draw=none]
      (97.87,112.59) --
      (146,113) --
      (146,189.54) --
      (173,190) --
      (173,129) --
      (244,129) --
      (244,170) --
      (280,170) --
      (280,149) --
      (319,149) --
      (319,35.8) --
      (291.65,35.8) --
      (127.58,34.5) --
      (97.5,34.5) --
      cycle;
    
    \fill[green!25, draw=none]
      (99,259) --
      (124.98,259) --
      (291.65,261) --
      (319,261) --
      (319,162) --
      (319,149) -- 
      (280,149) --
      (280,170) --
      (244,170) --
      (244,129) --
      (173,129) --
      (173,190) --
      (146,189.54) --
      (146,113) --
      (97.87,112.59) --
      cycle;
    
    \draw (128.63,56.87) node [anchor=north west][inner sep=0.75pt]  [font=\normalsize] [align=left] {\parbox{5cm}{$K_{ji} \neq 0\ \\
    ( H_{1j} =\dotsc =H_{ij} =0)$}};
    \draw (144.79,213) node [anchor=north west][inner sep=0.75pt]  [font=\normalsize] [align=left] {\parbox{5cm}{$H_{ij} \neq 0 \ ( K_{ji} =0)$}};

    \end{tikzpicture}
    
    \caption{$H, K$ structure.}
    \label{fig:HK}
    \end{subfigure}
    \caption{Coefficient matrix structure.}
    \label{fig:PQRHK}
\end{figure}
It is worthwhile mentioning that condition \eqref{a:PQR_lower} is weaker and easier to verify than the condition in \cite[Remark 3.1(i)]{DTT26}. Moreover, based on Figure~\ref{fig:PQRHK}, the coefficient matrices can be easily customized to meet specific computational requirements. In practice, $P$, $Q$, and $R$ are usually selected such that each column of $P$ and $Q$ (when $Q\neq0$), as well as each row of R, has exactly one nonzero entry equal $1$. Similarly, $H$ and $K$ can be constructed analogously to, or even chosen equal to, $P$ and $R$, respectively, as described in Section~\ref{s:realizations}.
\end{remark}

Now, we denote $\diag(\ell) :=\diag(\ell_1,\dots,\ell_p)$, set
\begin{align*}
\Omega &:= 2D - N - N^\top - MM^\top, \\
\Psi &:= (H-K^\top)\bL^*\DB\bL(H^\top -K), \\
\Upsilon &:=\begin{cases}
\frac{1}{2}(P-R^\top)\diag(\ell)(P^\top -R) &\text{if~} Q =0, \\
(P-Q)\diag(\ell)(P^\top -Q^\top) + (P-R^\top)\diag(\ell)(P^\top -R) &\text{otherwise},
\end{cases}
\end{align*}
and consider the following assumptions on the coefficient matrices.

\begin{assumption}[Standing assumptions]
\label{a:stand}
~
\begin{enumerate}
 \item\label{a:stand_kerM}
$\ker M^\top = \lspan\{\bOne\}$.
\item\label{a:stand_N}
$\sum_{i,j=1}^n N_{ij} = \bOne^\top N \bOne= \bOne^\top D \bOne= \sum_{i=1}^n \delta_i$.
\item\label{a:stand_PR}
$P^\top \bOne=R\bOne=\bOne$ if $\bC \neq 0$, and $P^\top=R=0$ if $\bC = 0$.
\item\label{a:stand_H}
$H^\top \bOne=\bOne$ if $\bB \neq 0$, and $H^\top=K=0$ if $\bB=0$.
\end{enumerate}
\end{assumption}
\begin{assumption}[Positive semidefiniteness]  
\label{a:PSD} 
There exists $\alpha \in [0,1)$ such that
\begin{align*}
\Omega + \alpha MM^\top - \frac{\gamma}{1+\alpha}\Psi -\gamma \Upsilon \succeq 0.    
\end{align*}
\end{assumption}

\begin{remark}
\label{r:OnAssump}
We make the following observations regarding Assumptions~\ref{a:stand} and \ref{a:PSD}.
\begin{enumerate}
\item\label{r:OnAssump_kerM}
It follows from Assumption~\ref{a:stand}\ref{a:stand_kerM} that 
\begin{align*}
\rank M &=\rank M^\top =n -\dim(\ker M^\top) =n -1, \\
\ran M &=(\ker M^\top)^\perp =\left\{(c_1, \dots, c_n)^\top\in \mathbb{R}^n: \bOne^\top c=\sum_{i=1}^n c_i = 0\right\}, \\
\ker(M^\top \otimes \Id) &=\ker M^\top \otimes \mathcal{H} =\Sigma, \\
\ran(M\otimes \Id) &=(\ker(M\otimes \Id)^\top)^\perp =(\ker(M^\top \otimes \Id))^\perp =\Sigma^\perp.
\end{align*}

\item\label{r:OnAssump_PR}
In view of Assumption~\ref{a:stand}\ref{a:stand_PR}, if $\bC\neq 0$, then we have for all $x \in \mathcal{H}$ that $P^\top (\bOne x) = R(\bOne x) = \bOne x$, and so
\begin{align*}
\Phi(\bOne x) =(P - Q)\bC R(\bOne x)  + Q\bC P^\top(\bOne x)  =(P - Q)\bC (\bOne x)  + Q\bC (\bOne x)  = P\bC(\bOne x),   
\end{align*} 
leading to $\bOne^\top \Phi(\bOne x) = \bOne^\top P\bC(\bOne x) = \bOne^\top\bC(\bOne x)$, which also holds when $\bC =0$.
\item\label{r:OnAssump_QK}
Assumptions~\ref{a:stand} and \ref{a:PSD} imply that $Q^\top\bOne = \bOne$ if $Q\neq 0$, and $K\bOne = \bOne$ if $\bB\neq 0$. Indeed, by Assumption~\ref{a:stand}, $\bOne^\top MM^\top \bOne =0$, $\bOne^\top(2D - N - N^\top)\bOne =0$,
\begin{align*}
\bOne^\top \Upsilon \bOne &=
\begin{cases}
0 &\text{if~} $Q =0$, \\
(\bOne - Q^\top\bOne)^\top \diag(\ell)(\bOne - Q^\top\bOne) \geq 0 &\text{otherwise},
\end{cases} \\
\text{and~}\bOne^\top \Psi \bOne &=
\begin{cases}
0 &\text{if~} \bB =0, \\
\bOne^\top \Psi \bOne = (\bOne - K\bOne)^\top \bL^* \DB\bL(\bOne-K\bOne) \succeq 0 &\text{otherwise}.
\end{cases}
\end{align*}
This together with Assumption~\ref{a:PSD} yields, for all $x\in \mathcal{H}$, 
\begin{align*}
\scal{x}{\bOne^\top\left(\Omega+\alpha MM^\top - \frac{\gamma}{1+\alpha}\Psi - \gamma \Upsilon\right)\bOne x} = -\frac{\gamma}{1+\alpha}\scal{x}{\bOne^\top \Psi\bOne x} - \gamma\scal{x}{\bOne^\top\Upsilon\bOne x} \geq 0.
\end{align*}
Therefore, $\bOne^\top\Upsilon\bOne = 0$ and $\langle x, \bOne^\top\Psi\bOne x\rangle= 0$, which implies that $Q^\top\bOne = \bOne$ if $Q\neq 0$, and $\bL K\bOne = \bL\bOne$ if $\bB\neq 0$. Since $\bL K \bOne = (\sum_{i=1}^n K_{1i} L_1, \dots, \sum_{i=1}^n K_{ri} L_r)^\top$ and $\bL \bOne = (L_1,\dots, L_r)^\top$, we deduce that $(\sum_{i=1}^n K_{1i}, \dots, \sum_{i=1}^n K_{ri})^\top = (1,\dots, 1)^\top$, equivalently, $K\bOne = \bOne$ if $\bB\neq 0$.

\item\label{r:OnAssump_UVX}
By Assumption~\ref{a:stand}\ref{a:stand_kerM}\&\ref{a:stand_PR}, if $\bC\neq 0$, then it follows from \cite[Remark 3.3(iii)]{DTT26} that $U =(P^\top -R)(M^\top )^\dag \in\mathbb{R}^{p\times m}$ solves $UM^\top =P^\top -R$ with minimal norm. Moreover, if additional $Q^\top \bOne =\bOne$, then $V=(P^\top -Q^\top )(M^\top )^\dag\in \mathbb{R}^{p\times m}$ solves $VM^\top =P^\top -Q^\top$ with minimal norm. When $\bC =0$, we have $P =0$ and $R =0$, in which case the associated matrix $U$ is simply $0$. 

Similarly, by Assumption~\ref{a:stand}\ref{a:stand_kerM}\&\ref{a:stand_H}, if $\bB\neq 0$ and $K\bOne = \bOne$, then $X=(H^\top-K)(M^\top)^\dagger \in \mathbb{R}^{r\times m}$ is the minimal norm solution of $XM^\top = H^\top - K$. When $\bB =0$, the associated matrix $X$ is simply $0$.
\end{enumerate}
\end{remark}

\subsection{Fixed-point encoding and preliminary results}

\begin{lemma}[Fixed points and Kuhn--Tucker points]
\label{l:fixzer}
Suppose Assumption~\ref{a:stand} holds. Then the following hold:
\begin{enumerate}
\item\label{l:fixzeros_T}
If $(\bz ,\bw ) \in \Fix T$ and $(\bx ,\by ) = S(\bz ,\bw )$, then $\bx=\bOne x\in\Sigma$, $\by=\bL(\bOne x)$, and $(x,\DB\bL K(\bOne x)-\bw)\in \bZ$.

\item\label{l:fixzeros_AB}
If $( x,\bs ) \in \bZ $, then there exists $\bz\in \mathcal{H}^m$ such that $(\bz ,\DB\bL K(\bOne x) -\bs) \in \Fix T$ and $(\bOne x ,\bL(\bOne x) ) = S(\bz ,\DB\bL K(\bOne x) -\bs)$.
\end{enumerate}
Consequently, $\Fix T \neq \varnothing$ if and only if $\bZ  \neq \varnothing$.
\end{lemma}

\begin{proof}
\ref{l:fixzeros_T}: Assume that $(\bz ,\bw ) \in \Fix T$ and $(\bx ,\by ) = S(\bz ,\bw )$. Then 
\begin{align*}
(\bz ,\bw )-T(\bz ,\bw )=(M^\top \bx ,\DB (\bL H^\top \bx -\by ))=(0, 0),
\end{align*}
and so $\bx \in \ker(M^\top \otimes \Id) =\Sigma$ due to  Remark~\ref{r:OnAssump}\ref{r:OnAssump_kerM}. Thus, $\bx=\bOne x$ for some $x\in\mathcal{H}$ and $\by=\bL H^\top\bx=\bL H^\top(\bOne x) = \bL(\bOne x)$.

Since $(\bx ,\by )=(\bOne x,\by) = S(\bz ,\bw )$, it holds that
\begin{align*}
D^{-1}\Big( M \bz + N (\bOne x)  - \gamma \Phi(\bOne x)  - \gamma H\bL^*(\DB\bL K(\bOne x)-\bw )\Big) &\in \bOne x + \gamma D^{-1}\bA (\bOne x),\\
    \bL K(\bOne x)- \DB^{-1}\bw + \bL H^\top(\bOne x) &\in \by + \DB^{-1}\bB\by,
\end{align*}
which leads to
\begin{align*}
    M\bz +N(\bOne x) -D(\bOne x)  &\in \gamma\bA (\bOne x)  + \gamma \Phi(\bOne x)  +\gamma H\bL^*(\DB\bL K(\bOne x)-\bw ),\\
    \DB\bL K(\bOne x)-\bw  &\in \bB\by = \bB\bL(\bOne x).
\end{align*}
Therefore, $M\bz +N(\bOne x) -D(\bOne x)\in \gamma\bA (\bOne x)  + \gamma \Phi(\bOne x)  +\gamma  H\bL^* \bB\bL(\bOne x)$. This together with Remark~\ref{r:OnAssump}\ref{r:OnAssump_PR} and Assumption~\ref{a:stand}\ref{a:stand_kerM}\ref{a:stand_N}\ref{a:stand_H} implies that
\begin{multline*}
\gamma\Big(\bOne^\top\bA (\bOne x) + \bOne^\top \bL^*\bB\bL (\bOne x) +
\bOne^\top\bC(\bOne x)\Big)=\gamma\Big(\bOne^\top\bA (\bOne x) + \bOne^\top H\bL^*\bB\bL(\bOne x) +
\bOne^\top \Phi(\bOne x)\Big)\\
\ni \bOne^\top M\bz + \bOne^\top N(\bOne x) -\bOne^\top D(\bOne x)
= \bOne^\top M\bz = 0^\top \bz = 0.
\end{multline*}
Hence, $(x,\DB\bL K(\bOne x)-\bw)\in\bZ$ as claimed.

\ref{l:fixzeros_AB}: Assume that $(x,\bs)\in\bZ$. Then $\bs\in\bB\bL(\bOne x)$ and $0\in \bOne^\top\bA(\bOne x) + \bOne^\top \bL^*\bs + \bOne^\top\bC(\bOne x)$. So, there exists $\ba\in\bA(\bOne x)$ such that
\begin{equation}\label{e:250903c}
0= \bOne^\top\ba + \bOne^\top \bL^*\bs + \bOne^\top\bC(\bOne x).
\end{equation}
Define $\bu:=\bOne x + \gamma D^{-1}\ba \in\bOne x + \gamma D^{-1}\bA(\bOne x)$ and
$\bv:= \bL(\bOne x)+ \DB^{-1}\bs
\in \bL(\bOne x) + \DB^{-1}\bB\bL(\bOne x)$.
It follows that
\begin{equation}\label{e:250903a}
\bOne x=J_{\gamma D^{-1}\bA}(\bu)
\text{~~and~~} 
\by:=\bL(\bOne x)=J_{\DB^{-1}\bB}(\bv).
\end{equation}
Define $\bw:= \DB\bL K(\bOne x) -\bs$, we have
\begin{equation}\label{e:250903b}
\bv = \bL(\bOne x) + \DB^{-1}\bs = \bL K(\bOne x) -\DB^{-1}\bw  + \bL (\bOne x)
\end{equation}
Next, we will find $\bz\in\mathcal{H}^m$ such that
\begin{equation}\label{e:250903d}
\bu  =D^{-1}\Big( M \bz  + N (\bOne x)  -\gamma \Phi(\bOne x)  - \gamma H\bL^*(\DB\bL K(\bOne x) -\bw )\Big),
\end{equation}
i.e., $M\bz =D\bu  - N(\bOne x)  + \gamma \Phi(\bOne x)  + \gamma H\bL^*(\DB\bL K(\bOne x) -\bw )$. Thus, to establish the existence of $\bz$, it suffices to prove that
\begin{equation*}
D\bu  - N(\bOne x)  + \gamma \Phi(\bOne x)  + \gamma H\bL^*(\DB\bL K(\bOne x) -\bw )
\in \ran(M\otimes \Id) = \Sigma^\perp
=\{\hat\bx\in\mathcal{H}^n: \bOne^\top\hat{\bx}=0\}.
\end{equation*}
Indeed, we check that
\begin{align*}
&\bOne^\top \Big(D\bu  - N (\bOne x)  + \gamma \Phi\bx  + \gamma H\bL^*(\DB\bL K(\bOne x) -\bw )\Big)\\
&=\bOne^\top D(\bOne x) + \gamma \bOne^\top\ba  - \bOne^\top N (\bOne x)  + \gamma \bOne^\top \Phi(\bOne x)  + \gamma \bOne^\top H\bL^*\bs\\
&=\gamma \Big(\bOne^\top\ba + \bOne^\top C(\bOne x) + \bOne^\top \bL^*\bs\Big)=0,
\end{align*}
due to $\bOne^\top D(\bOne x) =\bOne^\top N(\bOne x)$ by Assumption~\ref{a:stand}\ref{a:stand_N}, $\bOne^\top \Phi(\bOne x)=\bOne^\top C(\bOne x)$ by Remark~\ref{r:OnAssump}\ref{r:OnAssump_PR}, and \eqref{e:250903c}. Therefore, we have proved the existence of $\bz$ satisfying \eqref{e:250903d}.

Now, combining \eqref{e:250903a}, \eqref{e:250903b}, and \eqref{e:250903d}, we have that $(\bOne x,\by)=S(\bz,\bw)$. Finally, we check that $T(\bz ,\bw ) =(\bz ,\bw ) -(M^\top (\bOne x) , \DB\bL H^\top(\bOne x) -\DB\by )=(\bz,\bw)$, i.e., $(\bz,\bw)\in\Fix T$.
\end{proof}

The following lemma characterizes the cluster points of the sequence generated by Algorithm~\ref{algo:full}.

\begin{lemma}
\label{l:cluster}
Suppose Assumption~\ref{a:stand} holds. Let $(\bz ^t, \bw ^t)_{t\in \mathbb{N}}$ and $(\bx^t, \by ^t)_{t\in \mathbb{N}}$ be the sequences generated by Algorithm~\ref{algo:full} and suppose $(\Id-T)(\bz ^t,\bw ^t)\to 0$ as $t\to +\infty$. Then the following hold:
\begin{enumerate}
\item\label{l:cluster_aux}
For all $i, j\in \{1, \dots, n\}$ and $k\in \{1, \dots, r\}$, $x^t_i -x^t_j\to 0$ and $y^t_k -L_k x^t_j\to 0$ as $t\to +\infty$.
\item\label{l:cluster_main}
For every weak cluster point $(\bar{\bz}, \bar{\bw}, \bar{\bx}, \bar{\by})$ of $(\bz^t, \bw^t, \bx^t, \by^t)_{t\in\mathbb{N}}$, it holds that $(\bar{\bz},\bar{\bw})\in \Fix T$, $\bar{\bx} =\bOne \bar{x}$, and $\bar{\by} =\bL(\bOne \bar{x})$ with $(\bar{x}, \DB\bL K\bar{\bx} -\bar{\bw})\in \bZ$; moreover, if the first rows of $N$, $P$, $Q$, and $H$ are zeros, then $\bar{x} =J_{\frac{\gamma}{\delta_1} A_1}(\frac{1}{\delta_1} \sum_{j=1}^m M_{1j}\bar{z}_j)$.
\end{enumerate}
\end{lemma}
\begin{proof}
\ref{l:cluster_aux}: Let $c :=(c_1, \dots, c_n)\in \mathbb{R}^n$ be any vector such that $\sum_{i=1}^n c_i =0$. By Remark~\ref{r:OnAssump}\ref{r:OnAssump_kerM}, $c\in \ran M$, and so $c = Md$ for some $d\in \mathbb{R}^{m}$. As $t\to +\infty$, since $\bz ^t -T_1(\bz ^t, \bw^t)\to 0$, one has 
\begin{align*}
c^\top \bx^t = d^\top M^\top\bx^t = d^\top(\bz^t - T_1(\bz ^t,\bw ^t)) \to 0.
\end{align*}
For all $i,j\in \{1, \dots, n\}$ with $i\neq j$, choosing~$c_i =1, c_j =-1$, and all other entries equal to zeros, we derive
\begin{align}\label{eq:xlxi}
x^t_i -x^t_j =c^\top\bx^t\to 0.
\end{align}
Next, since $\by^t -\bL H^\top \bx^t =-\DB^{-1}(\bw^t -T_2(\bz ^t, \bw^t))\to 0$, using Assumption~\ref{a:stand}\ref{a:stand_H} and \eqref{eq:xlxi}, we have that, for all $j\in \{1, \dots, n\}$,
\begin{align*}
\by^t -\bL(\bOne x^t_j) =\by^t -\bL H^\top \bx^t +\bL H^\top(\bx^t -\bOne x^t_j)\to 0,    
\end{align*}
which completes the proof of \ref{l:cluster_aux}.

\ref{l:cluster_main}: For each $i\in \{1, \dots, n\}$, $\bx  =(x_1, \dots, x_n)\in \mathcal{H}^n$ and $x\in \mathcal{H}$, we set 
\begin{align*}
\Phi\bx:=(\Phi_1\bx, \dots, \Phi_n\bx)
\text{~~and~~} \phi_i x := \Phi_i(\bOne x),    
\end{align*}
where $\Phi$ is given by \eqref{e:Phi}, which means
\begin{align*}
\Phi_i\bx =\sum_{j=1}^{p}\left((P_{ij}-Q_{ij})C_j\left(\sum_{k=1}^n R_{jk}{x_k}\right)+Q_{ij}C_j\left(\sum_{k=1}^n P_{kj}x_k\right)\right),
\end{align*}
and, by Remark~\ref{r:OnAssump}\ref{r:OnAssump_PR},
\begin{align}
\label{eq:phi_C}
\phi_i x =\left(P\bC (\bOne x)\right)_i.
\end{align}
For each $t\in \mathbb{N}$, set $\bu ^t :=D^{-1}\Big(M\bz ^t +N \bx^t -\gamma H\bL^*(\DB\bL K\bx^t -\bw ^t)\Big)$ and $\bv^t :=\bL K\bx^t -\DB^{-1}\bw ^t +\bL H^\top \bx^t$. Then $\bx^t =J_{\gamma D^{-1} \bA}(\bu^t -\gamma D^{-1}\Phi\bx^t)$ and $\by^t =J_{\DB^{-1}\bB}(\bv^t)$, which are equivalent to 
\begin{align*}
\frac{\delta_i}{\gamma}(u_i^t -x_i^t) -\Phi_i\bx^t&\in A_i x_i^t, \quad i\in \{1, \dots, n\} \\
\text{and}\quad \eta_k (v_k^t-y_k^t)&\in B_k y_k^t, \quad k\in\{1,\dots,r\}.
\end{align*}
This can be written as
\begin{align*}
x^t_i -x^t_n&\in (A_i +\phi_i)^{-1}\left(\frac{\delta_i}{\gamma}(u^t_i -x^t_i) -\Phi_i\bx^t +\phi_i x^t_i\right) -x^t_n, \quad i\in \{1, \dots, n-1\}, \\
s^t&\in (A_n +\phi_n)(x^t_n) +\sum_{i=1}^{n-1} \left(\frac{\delta_i}{\gamma}(u^t_i -x^t_i) -\Phi_i\bx^t +\phi_i x^t_i\right) +\sum_{k=1}^r L_k^*(\eta_k (v^t_k -y^t_k)), \\
y^t_k -L_k x^t_n&\in B_k^{-1}(\eta_k (v_k^t-y_k^t)) -L_k x^t_n, \quad k\in\{1,\dots,r\},
\end{align*}
where 
\begin{align}\label{eq:s^t}
s^t :=\sum_{i=1}^n \frac{\delta_i}{\gamma}(u^t_i -x^t_i) -\sum_{i=1}^n (\Phi_i\bx^t -\phi_i x^t_i) +\sum_{k=1}^r L_k^*(\eta_k (v^t_k -y^t_k)).
\end{align}
Letting $\mathcal{S}\colon \mathcal{H}^{n+r} \rightrightarrows \mathcal{H}^{n+r}$ be given by
\begin{align*}
\mathcal{S} := \diag\left(\begin{bmatrix}
(A_1+\phi_1)^{-1} \\
\vdots \\
(A_{n-1}+\phi_{n-1})^{-1} \\
(A_n+\phi_n)\\
B_1^{-1}\\
\vdots\\
B_r^{-1}
\end{bmatrix}\right) 
+\begin{bmatrix}
0 & \dots & 0 & -\Id & 0 & \dots & 0\\
\vdots & \ddots & \vdots & \vdots & \vdots & \ddots & \vdots\\
0 & \dots & 0 & -\Id & 0 & \dots & 0\\
\Id & \dots & \Id & 0 & L_1^* & \dots & L_r^*\\
0 & \dots & 0 & -L_1 & 0 & \dots & 0\\
\vdots & \ddots & \vdots & \vdots & \vdots & \ddots & \vdots\\
0 & \dots & 0 & -L_r & 0 & \dots & 0\\
\end{bmatrix},    
\end{align*} 
we derive that
\begin{align}\label{eq:inclu}
\begin{bmatrix}
x_1^t -x_n^t \\
\vdots \\
x_{n-1}^t -x_n^t \\
s^t \\
y_1^t - L_1x_n^t\\
\vdots\\
y_r^t - L_r x_n^t
\end{bmatrix}
\in \mathcal{S}\left(\begin{bmatrix}
\frac{\delta_1}{\gamma}(u_1^t -x_1^t) -\Phi_1\bx^t +\phi_1 x_1^t \\
\vdots \\
\frac{\delta_{n-1}}{\gamma}(u_{n-1}^t -x_{n-1}^t) -\Phi_{n-1}\bx^t +\phi_{n-1} x_{n-1}^t\\
x_n^t\\
\eta_1 (v_1^t - y_1^t)\\
\vdots\\
\eta_r (v_r^t - y_r^t)
\end{bmatrix}\right).
\end{align}
As $C_1,\dots, C_p \colon \mathcal{H}\rightarrow\mathcal{H}$ are monotone and Lipschitz continuous, they are maximally monotone operators with full domain. By \eqref{eq:phi_C}, $\phi_1,\dots, \phi_n$ are also maximally monotone operators. Using \cite[Corollary 25.5(i), Example 20.35]{BC17}, $\mathcal{S}$ is a maximally monotone operator as the sum of a maximally monotone operator and a skew symmetric linear operator.

Now, we show that the left-hand side of \eqref{eq:inclu} converges strongly to $0$. In view of \ref{l:cluster_aux}, it suffices to show that $s^t\to 0$ as $t\to +\infty$. We note that, for each $i\in \{1, \dots, n\}$,
\begin{align}\label{eq:Phiphi}
\Phi_i\bx^t -\phi_i x^t_i =\Phi_i\bx^t -\Phi_i(\bOne x^t_i)\to 0 \text{~~as~} t\to +\infty
\end{align}
due to $\bx^t -\bOne x^t_i\to 0$ by \ref{l:cluster_aux} and $\Phi_i$ is Lipschitz continuous by the Lipschitz continuity of $C_1, \ldots,C_p$. By the definitions of $\bu^t$ and $\bv^t$, and Remark~\ref{r:OnAssump}\ref{r:OnAssump_kerM},
\begin{align*}
D\bu^t -N\bx^t +\gamma H\bL^*\DB(\bv^t -\bL H^\top \bx^t) =M\bz^t\in \ran(M\otimes \Id). 
\end{align*}
Noting from Assumption~\ref{a:stand}\ref{a:stand_H} that $\bOne^\top H =\bOne^\top$, we have $\bOne^\top(D\bu^t -N\bx^t) +\gamma \bOne^\top \bL^*\DB(\bv^t -\bL H^\top \bx^t) =0$, which implies that
\begin{align}\label{eq:uxvy}
&\sum_{i=1}^n \frac{\delta_i}{\gamma}(u^t_i -x^t_i) +\sum_{k=1}^r L_k^*(\eta_k (v^t_k -y^t_k)) \notag \\
&= \frac{1}{\gamma}\bOne^\top D(\bu^t -\bx^t) +\bOne^\top \bL^*\DB(\bv^t -\by^t) \notag \\
&= \frac{1}{\gamma} \bOne^\top (N\bx^t -D\bx^t) -\bOne^\top \bL^*\DB(\bv^t -\bL H^\top \bx^t) +\bOne^\top \bL^*\DB(\bv^t -\by^t) \notag \\
&= \frac{1}{\gamma}\bOne^\top (N -D)\bx^t +\bOne^\top \bL^*\DB(\bL H^\top\bx^t -\by^t) \notag \\
&= \frac{1}{\gamma}\bOne^\top (N -D)(\bx^t -\bOne x^t_n) + \bOne^\top \bL^*(\bw^t -T_2(\bz ^t, \bw^t))\to 0 \text{~~as~} t\to +\infty, 
\end{align}
where we use $\bOne^\top (N -D)(\bOne x^t_n) =0$ (by Assumption~\ref{a:stand}\ref{a:stand_N}) and $\bx^t -\bOne x^t_i\to 0$ (by \ref{l:cluster_aux}). Combining \eqref{eq:Phiphi} and \eqref{eq:uxvy} with \eqref{eq:s^t}, we deduce that $s^t\to 0$ as $t\to +\infty$.
    
Since $(\bar{\bz}, \bar{\bw}, \bar{\bx}, \bar{\by})$ is a weak cluster point of $(\bz^t, \bw^t, \bx^t, \by^t)_{t\in\mathbb{N}}$, there exists a subsequence, denoted by $(\bz^t, \bw^t, \bx^t, \by^t)_{t\in\mathbb{N}}$ without relabeling, such that $(\bz^t, \bw^t, \bx^t, \by^t)\rightharpoonup (\bar{\bz}, \bar{\bw}, \bar{\bx}, \bar{\by})$. It follows from \ref{l:cluster_aux} that $\bar{\bx} =\bOne \bar{x}$ and $\bar{\by} =\bL(\bOne \bar{x})$ for some $\bar{x}\in\mathcal{H}$. We also have that 
\begin{align*}
&\bu^t\rightharpoonup \bar{\bu} =D^{-1}(M\bar{\bz} +N\bar{\bx} -\gamma H\bL^*(\DB\bL K\bar{\bx} -\bar{\bw})) \\ 
\text{and~} &\bv^t\rightharpoonup \bar{\bv} =\bL K\bar{\bx} -\DB^{-1}\bar{\bw} +\bL H^\top \bar{\bx} =\bL K\bar{\bx} -\DB^{-1}\bar{\bw} +\bL (\bOne \bar{x}).    
\end{align*}
By combining with \eqref{eq:uxvy}, 
\begin{align}\label{eq:uxvy0}
\sum_{i=1}^n \frac{\delta_i}{\gamma}(\bar{u}_i -\bar{x}) +\sum_{k=1}^r L_k^*\eta_k(\bar{v}_k -\bar{y}_k) =0.
\end{align}
Since the graph of a maximally monotone operator is sequentially closed in the weak-strong topology \cite[Proposition~20.38(ii)]{BC17}, passing to the limit in \eqref{eq:inclu} and using \ref{l:cluster_aux} and \eqref{eq:Phiphi}, we obtain that
\begin{align*}
\begin{bmatrix}
0 \\
\vdots \\
0 \\
0\\
0\\
\vdots \\
0
\end{bmatrix}
\in
\mathcal{S}\left(\begin{bmatrix}
\frac{\delta_1}{\gamma}(\bar{u}_1 -\bar{x})\\
\vdots \\
\frac{\delta_{n-1}}{\gamma}(\bar{u}_{n-1} -\bar{x})\\
\bar{x}\\
\eta_1(\bar{v}_1 - \bar{y}_1)\\
\vdots\\
\eta_r(\bar{v}_r - \bar{y}_r)
\end{bmatrix}\right),
\end{align*}
which together with \eqref{eq:uxvy0} yields
\begin{align}\label{eq:sol}
\begin{aligned}
\frac{\delta_i}{\gamma}(\bar{u}_i -\bar{x}) - \phi_i \bar{x} &\in A_i \bar{x}, \quad i\in \{1, \dots, n-1\}, \\
\frac{\delta_n}{\gamma}(\bar{u}_n - \bar{x}) -\phi_n\bar{x} = -\sum_{i=1}^{n-1} \frac{\delta_i}{\gamma}(\bar{u}_i - \bar{x})-\sum_{k=1}^{r} L_k^*\eta_k(\bar{v}_k - \bar{y}_k) - \phi_n\bar{x} &\in A_n\bar{x},\\
\eta_k(\bar{v}_k - \bar{y}_k) &\in B_k L_k \bar{x}, \quad k\in\{1,\dots,r\},
\end{aligned}
\end{align}
Therefore,
\begin{align}
\bar{x} &= J_{\frac{\gamma}{\delta_i}A_i}\Big(\bar{u}_i - \frac{\gamma}{\delta_i}\phi_i\bar{x}\Big) =J_{\frac{\gamma}{\delta_i}A_i}\Big(\bar{u}_i - \frac{\gamma}{\delta_i}\Phi_i\bar{\bx}\Big), \quad i\in \{1, \dots, n\}, \label{eq:xbar} \\
\bar{y}_k = L_k\bar{x} &= J_{\frac{1}{\eta_k}B_k}(\bar{v}_k), \quad k\in\{1,\dots,r\}, \notag
\end{align}
which implies that $(\bar{\bx},\bar{\by}) = S(\bar{\bz}, \bar{\bw})$, and thus $T(\bar{\bz},\bar{\bw})= (\bar{\bz} - M^\top(\bOne\bar{x}), \bar{\bw} - \DB(\bL(\bOne \bar{x})-\bar{\by})=(\bar{\bz}, \bar{\bw})$. As a result, $(\bar{\bz},\bar{\bw})\in \Fix T$. 

Next, we have from \eqref{eq:phi_C} and Assumption~\ref{a:stand}\ref{a:stand_PR} that $\sum_{i=1}^{n} \phi_i\bar{x} =\sum_{j=1}^p C_j \bar{x}$. By summing up the relations in \eqref{eq:sol}, we derive
\begin{align*}
    -\bOne^\top \bL^* \DB(\bar{\bv} - \bar{\by }) - \bOne^\top \bC(\bOne\bar{x}) &\in \bOne^\top \bA(\bOne\bar{x}), \\
    \DB(\bar{\bv} - \bar{\by })&\in \bB \bL (\bOne\bar{x}).
\end{align*}
Since $\DB(\bar{\bv} -\bar{\by}) =\DB\bL K\bar{\bx } -\bar{\bw} +\DB(\bL(\bOne \bar{x}) -\bar{\by}) =\DB\bL K\bar{\bx } -\bar{\bw}$, it follows that $(\bar{x}, \DB\bL K\bar{\bx }-\bar{\bw }) \in \bZ $. Finally, by \eqref{eq:xbar}, if the first rows of $N$, $P$, $Q$, and $H$ are all zeros, then $\bar{x} =J_{\frac{\gamma}{\delta_1} A_1}(\bar{u}_1 - \frac{\gamma}{\delta_1}\phi_1\bar{x}) =J_{\frac{\gamma}{\delta_1} A_1}(\frac{1}{\delta_1} \sum_{j=1}^m M_{1j}\bar{z}_j)$.
\end{proof}

Next, we establish the convergence of Algorithm~\ref{algo:full} in the presence of Fej\'er monotonicity.

\begin{lemma}
\label{l:wcvg}
Suppose Assumption~\ref{a:stand} holds and the first rows of $N$, $P$, $Q$, and $H$ are entirely zeros. Let $(\bz ^t, \bw ^t)_{t\in \mathbb{N}}$ and $(\bx^t, \by ^t)_{t\in \mathbb{N}}$ be the sequences generated by Algorithm~\ref{algo:full} and suppose that $(\Id-T)(\bz ^t,\bw ^t)\to 0$ as $t\to +\infty$. Then the following hold:
\begin{enumerate}
\item\label{l:wcvg_bounded}
If $(\bz ^t, \bw ^t)_{t\in \mathbb{N}}$ is bounded, then so is $(\bx^t, \by ^t)_{t\in \mathbb{N}}$.
\item\label{l:wcvg_Fejer}
If $(\bz^t, \bw^t)_{t\in \mathbb{N}}$ is Fej\'er monotone with respect to $\Fix T$, then $(\bz^t,\bw^t)\rightharpoonup (\bar{\bz}, \bar{\bw})\in \Fix T$ as $t\to +\infty$.
\item\label{l:wcvg_shadow}
As $t\to +\infty$, if $(\bz^t,\bw^t)\rightharpoonup (\bar{\bz}, \bar{\bw})$, then $(\bx^t,\by ^t) \rightharpoonup (\bar{\bx }, \bar{\by })$ and $(\DB\bL K\bx^t-\bw^t)\rightharpoonup (\DB\bL K\bar{\bx }-\bar{\bw })$ with $\bar{\bx} =\bOne\bar{x}$, $\bar{\by} =\bL(\bOne\bar{x})$, and $(\bar{x},\DB\bL K\bar{\bx}-\bar{\bw})\in \bZ$.
\end{enumerate}
\end{lemma}
\begin{proof}
\ref{l:wcvg_bounded}: Since the first rows of $N$, $P$, $Q$, and $H$ are all zeros, we have from \eqref{eq:explicit_x1} that $x^t_1 =J_{\frac{\gamma}{\delta_1}A_1}\Big(\frac{1}{\delta_1}\sum_{j=1}^m M_{1j}z^t_j\Big)$ and $x^0_1 =J_{\frac{\gamma}{\delta_1}A_1}\Big(\frac{1}{\delta_1}\sum_{j=1}^m M_{1j}z^0_j\Big)$. By the nonexpansiveness of $J_{\frac{\gamma}{\delta_1}A_1}$,
\begin{align*}
\|x^t_1 -x^0_1\|\leq \left\|\frac{1}{\delta_1}\sum_{j=1}^m M_{1j}z^t_j -\frac{1}{\delta_1}\sum_{j=1}^m M_{1j}z^0_j\right\|.    
\end{align*}
Since, for each $j\in \{1, \dots, m\}$, $(z^t_j)_{t\in \mathbb{N}}$ is bounded, so is $(x_1^t)_{t\in \mathbb{N}}$, which together with Lemma~\ref{l:cluster}\ref{l:cluster_aux} implies the boundedness of $(\bx^t)_{t\in \mathbb{N}}$. The boundedness of $(\by^t)_{t\in \mathbb{N}}$ then follows from the nonexpansiveness of $J_{\DB^{-1} \bB}$ along with the boundedness of $(\bw^t)_{t\in\mathbb{N}}$, $(\bx^t)_{t\in\mathbb{N}}$, and $\bL$.

\ref{l:wcvg_Fejer}: As $(\bz^t, \bw^t)_{t\in \mathbb{N}}$ is Fej\'er monotone, it is bounded due to \cite[Proposition~5.4(i)]{BC17}. By \ref{l:wcvg_bounded}, $(\bx^t, \by ^t)_{t\in \mathbb{N}}$ is bounded.

Now, let $(\bar{\bz}, \bar{\bw}) =(\bar{z}_1, \dots, \bar{z}_m, \bar{w}_1, \dots, \bar{w}_r)\in \mathcal{H}^m\times \prod_{k=1}^r \mathcal{G}_k$ be an arbitrary weak cluster point of $(\bz^t, \bw^t)_{t\in \mathbb{N}}$. By passing to another subsequences if necessary, there exist $\bar{\mathbf{x}} \in \mathcal{H}^n$ and $\bar{\by}\in \prod_{k=1}^r \mathcal{G}_k$ such that $(\bar{\bz}, \bar{\bw}, \bar{\bx}, \bar{\by})$ is a weak cluster point of $(\bz^t, \bw^t, \bx^t, \by^t)_{t\in\mathbb{N}}$. According to Lemma~\ref{l:cluster}\ref{l:cluster_main}, $(\bar{\bz},\bar{\bw})\in \Fix T$. By \cite[Theorem 5.5]{BC17}, $(\bz^t, \bw^t)_{t\in\mathbb{N}}$ converges weakly to $(\bar{\bz}, \bar{\bw})\in \Fix T$.

\ref{l:wcvg_shadow}: Since $(\bz^t,\bw^t)_{t\in \mathbb{N}}$ converges weakly to $(\bar{\bz}, \bar{\bw})$, it is bounded, so $(\bx^t,\by^t)_{t\in \mathbb{N}}$ is also bounded due to \ref{l:wcvg_bounded}. Let $(\bar{\bx},\bar{\by})\in\mathcal{H}^n\times \prod_{k=1}^r \mathcal{G}_k$ be an arbitrary weak cluster point of $(\bx^t, \by^t)_{t\in\mathbb{N}}$. Then $(\bar{\bz}, \bar{\bw}, \bar{\bx}, \bar{\by})$ is a weak cluster point of $(\bz^t, \bw^t, \bx^t, \by^t)_{t\in\mathbb{N}}$. By Lemma~\ref{l:cluster}\ref{l:cluster_main}, $\bar{\bx} =\bOne \bar{x}$, $\bar{\by} =\bL(\bOne \bar{x})$, $(\bar{x}, \DB\bL K\bar{\bx} -\bar{\bw})\in \bZ$, and $\bar{x} =J_{\frac{\gamma}{\delta_1} A_1}(\frac{1}{\delta_1} \sum_{j=1}^m M_{1j}\bar{z}_j)$. Thus, $\bar{\bx}$ and $\bar{\by}$ do not depend on the choice of cluster points, and we must have $(\bx^t, \by^t) \rightharpoonup (\bar{\bx},\bar{\by})$. Together with the weak convergence of $(\bw^t)_{t\in\mathbb{N}}$, we obtain that $(\DB\bL K\bx^t-\bw^t)\rightharpoonup (\DB\bL K\bar{\bx }-\bar{\bw })$, which completes the proof.  
\end{proof}

We observe that $T=\Id - \Gamma$ with $\Gamma:\mathcal{H}^m\times\prod_{k=1}^r\mathcal{G}_k\to \mathcal{H}^m\times\prod_{k=1}^r \mathcal{G}_k$:
\begin{align}
\label{e:defGG}
(\bz, \bw)
\mapsto
\begin{bmatrix}
M^\top & 0\\
\DB\bL H^\top & - \DB
\end{bmatrix}
S(\bz,\bw)
=\begin{bmatrix}
M^\top & 0\\
\DB\bL H^\top & - \DB
\end{bmatrix}
\begin{bmatrix}
\bx\\
\by
\end{bmatrix},
\end{align}
where $(\bx,\by)=S(\bz,\bw)$. Thus, the algorithm can be written as
\begin{align*}
(\bz^{t+1},\bw^{t+1})=
(\bz^{t},\bw^{t}) - \lambda_t \Gamma(\bz^{t},\bw^{t}).
\end{align*}
Our aim is to prove that $T$ is conically quasiaveraged in the space $\mathcal{H}^m \times \prod_{k=1}^r\mathcal{G}_k$ with the (scaled) inner product
\begin{align}
\scal{(\bz ,\bw )}{(\bar\bz, \bar\bw)}_{\star} := \scal{\bz}{\bar\bz } + \gamma\scal{\DB^{-1}\bw}{\bar\bw}
\end{align}
for $(\bz ,\bw ),(\bar{\bz },\bar{\bw })\in\mathcal{H}^m\times \prod_{k=1}^r\mathcal{G}_k$ and the corresponding induced $\star$-norm
\begin{align*}
\|(\bz,\bw)\|_\star=\sqrt{\scal{(\bz ,\bw )}{(\bz, \bw)}_{\star}}.
\end{align*}

In view of Proposition~\ref{p:quasi_co_averaged}, we will show that $\Gamma$ is quasicomonotone. Thus, the next technical lemma is a key ingredient in our analysis.

\begin{lemma}[A metric inequality for $\Gamma$]
\label{l:GG_ineq}
Let $(\bz,\bw), (\bar\bz,\bar\bw)\in \mathcal{H}^{m}\times \prod_{k=1}^r \mathcal{G}_k$, $(\bx,\by) =S(\bz,\bw)$, $(\bar\bx,\bar\by) =S(\bar\bz,\bar\bw)$, and $\alpha\in \mathbb{R}$. Then
\begin{align*}
&\scal{\Gamma(\bz,\bw)-\Gamma(\bar\bz,\bar\bw)}{(\bz-\bar\bz,\bw-\bar\bw)}_{\star} \notag\\
&\geq \frac{1 -\alpha}{2}\|\Gamma(\bz,\bw)-\Gamma(\bar\bz,\bar\bw)\|_{\star}^2 + \gamma\scal{\bx - \bar\bx}{\Phi\bx-\Phi\bar\bx} +\frac{1}{2}\scal{\bx -\bar\bx}{(\Omega +\alpha MM^\top - \frac{\gamma}{1+\alpha}\Psi)(\bx - \bar\bx)} \notag\\
&\qquad + \frac{\gamma}{2(1+\alpha)}\left\|\alpha\left(\sqrt{\DB}\bL H^\top (\bx-\bar\bx) -\sqrt{\DB}(\by-\bar\by)\right) +\sqrt{\DB} \bL K(\bx-\bar\bx) -\sqrt{\DB} (\by-\bar\by)\right\|^2. 
\end{align*}
\end{lemma}
\begin{proof}
Set $\Delta\bx :=\bx -\bar\bx$, $\Delta\by :=\by -\bar\by$, $\Delta\bz :=\bz -\bar\bz$, and $\Delta\bw :=\bw -\bar\bw$. Then
\begin{align}
\scal{\Gamma(\bz,\bw)-\Gamma(\bar\bz,\bar\bw)}{(\Delta\bz,\Delta\bw)}_{\star}
&=\scal{M^\top \Delta\bx}{\Delta\bz}
+\gamma\scal{(\bL H^\top \Delta\bx - \Delta\by)}{\Delta\bw}.
\label{e:Gamma_inner_prod}
\end{align}
By the definition of $S$ in \eqref{e:sol_oper_S}, we have
$\bx =J_{\gamma D^{-1}\bA }(\bu)$ and $\bar{\bx }=J_{\gamma D^{-1}\bA }(\bar{\bu })$, where 
\begin{align*}
\bu &:= D^{-1}\Big(M\bz + N\bx -\gamma \Phi\bx -\gamma H\bL^*(\DB\bL K\bx-\bw)\Big),\\
\bar{\bu} &:= D^{-1}\Big(M\bar{\bz} + N\bar{\bx} -\gamma \Phi\bar{\bx} -\gamma H\bL^*(\DB\bL K\bar{\bx}-\bar{\bw})\Big).
\end{align*}
So, $D\bu -D\bx  \in \gamma\bA \bx $ and $D\bar{\bu }-D\bar{\bx } \in \gamma\bA \bar{\bx}$. Set $\Delta\bu =\bu -\bar\bu$. Since $\gamma \bA$ is monotone,
\begin{align*}
0 &\leq \scal{\Delta\bx}{D\Delta\bu -D\Delta\bx} \notag \\ 
&= \scal{\Delta\bx}{M\Delta\bz  + N\Delta\bx  - \gamma \Phi\Delta\bx - \gamma H\bL^*(\DB\bL K\Delta\bx -\Delta\bw ) - D\Delta\bx}\\
&=\scal{M^\top \Delta\bx}{\Delta\bz}
+
\scal{\Delta\bx}{(N - D)\Delta\bx}
- \gamma \scal{\Delta\bx}{\Phi\Delta\bx} \notag \\
&\quad - \gamma \scal{\bL H^\top\Delta\bx}{\DB\bL K\Delta\bx}
+ \gamma\scal{\bL H^\top\Delta\bx}{\Delta\bw}.
\end{align*}
Rearranging terms and using \eqref{e:Gamma_inner_prod}, we obtain
\begin{align}
\label{e:Gamma_inner_prod_2}
\scal{\Gamma(\bz,\bw)-\Gamma(\bar\bz,\bar\bw)}{(\Delta\bz,\Delta\bw)}_{\star}
&\geq \scal{\Delta\bx}{(D - N)\Delta\bx} + \gamma \scal{\Delta\bx}{\Phi\Delta\bx} \notag \\
&\qquad + \gamma \scal{\sqrt{\DB}\bL H^\top\Delta\bx}{\sqrt{\DB}\bL K\Delta\bx}
- \gamma\scal{\Delta\by}{\Delta\bw}.
\end{align}
The first term on the right-hand side of \eqref{e:Gamma_inner_prod_2} is rewritten as
\begin{align}
\scal{\Delta\bx}{(D - N)\Delta\bx}
&= \frac{1}{2}\scal{\Delta\bx}{[2D - 2N - MM^\top +\alpha MM^\top]\Delta\bx} + \frac{1-\alpha}{2}\|M^\top \Delta\bx\|^2\notag\\
&= \frac{1}{2}\scal{\Delta\bx}{[2D - N - N^\top - MM^\top +\alpha MM^\top]\Delta\bx} \notag \\
&\qquad + \frac{1-\alpha}{2}\|\Gamma(\bz,\bw) -\Gamma(\bar\bz,\bar\bw)\|_{\star}^2 - \frac{\gamma(1-\alpha)}{2}\|\sqrt{\DB}\bL H^\top\Delta\bx -\sqrt{\DB}\Delta\by\|^2,
\label{e:Gamma_inner_prod_2_term1}
\end{align}
where the last identity follows directly from the definition of $\Gamma$ and $\star$-norm. Next, set $\bv:=\bL K\bx -\DB^{-1}\bw  + \bL H^\top \bx$, $\bar{\bv}:=\bL K\bar{\bx }-\DB^{-1}\bar{\bw } + \bL H^\top \bar{\bx }$. Then $\by =J_{\DB^{-1}\bB }(\bv)$ and $\bar{\by }=J_{\DB^{-1}\bB }(\bar{\bv})$, or equivalently, $\DB(\bv-\by) \in \bB \by $ and $\DB(\bar{\bv}-\bar{\by})\in \bB \bar{\by}$. Since $\gamma>0$ and $\bB$ is monotone,
\begin{align*}
0 &\leq \gamma \scal{\by -\bar\by}{\DB(\bv -\bar\bv) -\DB(\by -\bar\by)} \notag\\
&= \gamma\scal{\Delta\by}{\DB\bL K\Delta\bx} - \gamma \scal{\Delta\by}{\Delta\bw} + 
\gamma\scal{\Delta\by}{\DB\bL H^\top \Delta\bx}
- \gamma\|\sqrt{\DB}\Delta\by\|^2.
\end{align*}
Thus, the fourth term on the right-hand side of \eqref{e:Gamma_inner_prod_2} can be estimated as
\begin{align*}
-\gamma\scal{\Delta\by}{\Delta\bw}
&\geq -\gamma\scal{\sqrt{\DB}\Delta\by}{\sqrt{\DB}\bL K\Delta\bx} - \gamma \scal{\sqrt{\DB}\Delta\by}{\sqrt{\DB}\bL H^\top \Delta\bx}
+\gamma\|\sqrt{\DB}\Delta\by\|^2.
\end{align*}
Adding the third terms on the right-hand sides of \eqref{e:Gamma_inner_prod_2} and \eqref{e:Gamma_inner_prod_2_term1} to this inequality, and setting $\mathbf{a} :=\sqrt{\DB}\bL H^\top\Delta\bx$, $\mathbf{b} :=\sqrt{\DB}\bL K\Delta\bx$, and $\mathbf{c} :=\sqrt{\DB}\Delta\by$, we obtain
\begin{align*}
&\gamma \scal{\sqrt{\DB}\bL H^\top\Delta\bx}{\sqrt{\DB}\bL K\Delta\bx} -\gamma \scal{\Delta\by}{\Delta\bw} - \frac{\gamma(1-\alpha)}{2}\|\sqrt{\DB}\bL H^\top\Delta\bx -\sqrt{\DB}\Delta\by\|^2 \notag\\
&\geq \gamma\left(\scal{\mathbf{a}}{\mathbf{b}} -\scal{\mathbf{c}}{\mathbf{b}} - \scal{\mathbf{c}}{\mathbf{a}} + \|\mathbf{c}\|^2\right) - \frac{\gamma(1-\alpha)}{2}\|\mathbf{a} -\mathbf{c}\|^2 \notag\\
&= -\gamma \left(\scal{\mathbf{a} - \mathbf{c}}{\frac{1-\alpha}{2}\left(\mathbf{a} - \mathbf{c}\right) - \left(\mathbf{b} - \mathbf{c}\right)}\right) \notag\\
&= -\frac{2\gamma}{1+\alpha} \scal{\frac{1+\alpha}{2}(\mathbf{a} - \mathbf{c})}{\frac{1-\alpha}{2}\left(\mathbf{a} - \mathbf{c}\right) - \left(\mathbf{b} - \mathbf{c}\right)} \notag \\
&= \frac{-\gamma}{2(1+\alpha)}\left(\|\mathbf{a} - \mathbf{b}\|^2 - \|\alpha(\mathbf{a} - \mathbf{c}) + (\mathbf{b} - \mathbf{c})\|^2\right),
\end{align*}
which together with \eqref{e:Gamma_inner_prod_2} and \eqref{e:Gamma_inner_prod_2_term1} completes the proof.
\end{proof}

\subsection{Convergence with or without cocoercivity}

\begin{lemma}[Conical quasiaveragedness]
\label{l:quasiaveraged}
Suppose Assumption~\ref{a:stand} holds. Let $(\bz,\bw)\in \mathcal{H}^{m}\times \prod_{k=1}^r \mathcal{G}_k$, $(\bar\bz,\bar\bw)\in \zer\Gamma =\Fix T$, $(\bx,\by) =S(\bz,\bw)$, $(\bar\bx,\bar\by) =S(\bar\bz,\bar\bw)$, and $\alpha\in [0, 1)$. Then
\begin{multline}
\scal{\Gamma(\bz,\bw)}{(\bz-\bar\bz,\bw-\bar\bw)}_{\star} \\
\geq \frac{1 -\alpha}{2}\|\Gamma(\bz,\bw)\|_{\star}^2 +\frac{1}{2}\scal{\bx -\bar\bx}{(\Omega +\alpha MM^\top - \frac{\gamma}{1+\alpha}\Psi -\gamma\Upsilon)(\bx - \bar\bx)} \\  
+ \frac{\gamma}{2(1+\alpha)}\left\|\alpha\left(\sqrt{\DB}\bL H^\top (\bx-\bar\bx) -\sqrt{\DB}(\by-\bar\by)\right) +\sqrt{\DB} \bL K(\bx-\bar\bx) -\sqrt{\DB} (\by-\bar\by)\right\|^2. \label{ineq:quasiaveraged}
\end{multline}
If Assumption~\ref{a:PSD} holds, then $\Gamma$ is $\frac{1-\alpha}{2}$-quasicomonotone and $T$ is conically $\frac{1}{1-\alpha}$-quasiaveraged.
\end{lemma}
\begin{proof}
By Lemma~\ref{l:fixzer}\ref{l:fixzeros_T}, $\bar\bx=\bOne \bar{x}$ for some $\bar{x}\in\mathcal{H}$. It then follows from Remark~\ref{r:OnAssump}\ref{r:OnAssump_PR} that $P^\top \bar{\bx} =R\bar{\bx} =\bar{\bx}$. Therefore,
\begin{align}\label{eq:xPhi}
&\scal{\bx-\bar{\bx}}{\Phi\bx - \Phi\bar{\bx}} \notag \\
&= \scal{P^\top(\bx - \bar{\bx})}{\bC R\bx - \bC R\bar{\bx}} +\scal{Q^\top(\bx -\bar{\bx})}{ [\bC P^\top \bx - \bC P^\top \bar{\bx}]
-[\bC R\bx - \bC R\bar{\bx}]} \notag \\
&= \scal{P^\top(\bx - \bar{\bx})}{\bC R\bx - \bC P^\top \bar{\bx}} + \scal{Q^\top(\bx -\bar{\bx})}{ \bC P^\top \bx -\bC R\bx} \notag \\
&= \scal{P^\top (\bx -\bar{\bx})}{\bC P^\top \bx -\bC P^\top \bar{\bx}} +\scal{(Q^\top-P^\top)(\bx-\bar{\bx})}{\bC P^\top\bx - \bC R\bx} \notag \\
&\geq \scal{(Q^\top-P^\top)(\bx-\bar{\bx})}{\bC P^\top\bx - \bC R\bx},
\end{align}
using the monotonicity of $\bC$.

Let $P^\top_j$, $Q^\top_j$, and $R_j$ denote the $j$-th rows of the matrices $P^\top$, $Q^\top$, and $R$, respectively. By the Cauchy--Schwarz inequality, the Lipschitz continuity of the $C_j$, and the AM-GM inequality, we derive that
\begin{align}
&\scal{(Q^\top-P^\top)(\bx-\bar{\bx})}{\bC P^\top\bx - \bC R\bx} \notag \\
&= \sum_{j=1}^p \scal{(Q^\top_j -P^\top_j)(\bx -\bar{\bx})}{C_j P^\top_j\bx -C_j R_j\bx} \notag \\
&\geq -\sum_{j=1}^p \|(P^\top_j -Q^\top_j)(\bx -\bar{\bx})\|\cdot \|C_j P^\top_j\bx -C_j R_j\bx\| \notag \\
&\geq -\sum_{j=1}^p \ell_j \|(P^\top_j - Q^\top_j)(\bx -\bar{\bx})\|\cdot \|(P^\top_j - R_j)\bx\| \notag \\
&\geq -\frac{1}{2}\sum_{j=1}^p \ell_j\left(\|(P^\top_j -Q^\top_j )(\bx-\bar{\bx})\|^2 + \|(P^\top_j -R_j)\bx\|^2\right) \notag \\
&= -\frac{1}{2}\left(\|\diag(\sqrt{\ell})(P^\top -Q^\top)(\bx-\bar{\bx}) \|^2 + \|\diag(\sqrt{\ell})(P^\top -R)(\bx-\bar{\bx})\|^2\right) \notag \\
&= -\frac{1}{2}\Big(\scal{\bx-\bar{\bx}}{(P-Q)\diag(\ell)(P^\top -Q^\top)(\bx-\bar{\bx})} \notag \\ 
&\qquad +\scal{\bx-\bar{\bx}}{(P-R^\top)\diag(\ell)(P^\top -R)(\bx-\bar{\bx})}\Big), \label{eq:xPhi++}
\end{align}
where the second last equality uses $(P^\top -R)\bar{\bx} =0$. Next, it follows from \eqref{eq:xPhi} and \eqref{eq:xPhi++} that 
\begin{align*}
\scal{\bx-\bar{\bx}}{\Phi\bx - \Phi\bar{\bx}} &\geq -\frac{1}{2}\scal{\bx-\bar{\bx}}{\Upsilon(\bx-\bar{\bx})}.  
\end{align*}
Substituting this into the inequality in Lemma~\ref{l:GG_ineq} and using the fact that $(\bar\bz,\bar\bw)\in \zer\Gamma =\Fix T$, we obtain \eqref{ineq:quasiaveraged}. Since $\alpha \in [0, 1)$, the quasicomonotonicity of $\Gamma$ follows from \eqref{ineq:quasiaveraged}, and the conical quasiaveragedness of $T =\Id -\Gamma$ then follows from Proposition~\ref{p:quasi_co_averaged}.
\end{proof}

\begin{lemma}[Conical averagedness under cocoercivity]
\label{l:averaged}
Suppose, for each $j\in \{1, \dots, p\}$, $C_j$ is $\frac{1}{\ell_j}$-cocoercive, Assumption~\ref{a:stand} holds, and $Q =0$. Let $(\bz,\bw), (\bar\bz,\bar\bw)\in \mathcal{H}^{m}\times \prod_{k=1}^r \mathcal{G}_k$, $(\bx,\by) =S(\bz,\bw)$, $(\bar\bx,\bar\by) =S(\bar\bz,\bar\bw)$, and $\alpha \in [0,1)$. Then
\begin{multline}
\scal{\Gamma(\bz,\bw)-\Gamma(\bar\bz,\bar\bw)}{(\bz-\bar\bz,\bw-\bar\bw)}_{\star} \\
\geq \frac{1-\alpha}{2}\|\Gamma(\bz,\bw)-\Gamma(\bar\bz,\bar\bw)\|_{\star}^2 
+ \frac{1}{2}\scal{\bx -\bar\bx}{\left(\Omega + \alpha MM^\top - \frac{\gamma}{1+\alpha}\Psi - \gamma\Upsilon\right)(\bx - \bar\bx)} \\
+ \frac{\gamma}{2(1+\alpha)}\left\|\alpha\left(\sqrt{\DB}\bL H^\top (\bx-\bar\bx) -\sqrt{\DB}(\by-\bar\by)\right) +\sqrt{\DB} \bL K(\bx-\bar\bx) -\sqrt{\DB} (\by-\bar\by)\right\|^2 \\
+ \frac{\gamma}{4}\left\|(P^\top  -R)(\bx  -\bar{\bx }) +2\diag(\ell)^{-1}(\bC R\bx -\bC R\bar{\bx})\right\|^2_{\diag(\ell)}. \label{ineq:averaged} 
\end{multline}
If Assumption~\ref{a:PSD} holds, then $\Gamma$ is $\frac{1-\alpha}{2}$-comonotone and $T$ is conically $\frac{1}{1-\alpha}$-averaged.
\end{lemma}
\begin{proof}
Since $Q=0$, using the equality $\scal{a}{b} =-\frac{\ell_j}{4}\|a\|^2 +\frac{\ell_j}{4}\Big\|a +\frac{2}{\ell_j}b\Big\|^2 -\frac{1}{\ell_j}\|b\|^2$ and the cocoercivity of the $C_j$, we have the estimation
\begin{align*}
&\scal{\bx-\bar{\bx}}{\Phi\bx - \Phi\bar{\bx}} = \scal{\bx  - \bar{\bx }}{P\bC R\bx  - P\bC R\bar{\bx }}  \notag \\
&= \scal{P^\top (\bx -\bar{\bx })}{\bC R\bx -\bC R\bar{\bx }} \notag \\
&= \scal{(P^\top -R)(\bx  -\bar{\bx })}{\bC R\bx-\bC R\bar{\bx }} + \scal{R\bx  - R\bar{\bx }}{\bC R\bx -\bC R\bar{\bx }} \notag \\
&= \sum_{j=1}^p\scal{(P^\top_j-R_j)(\bx  -\bar{\bx })}{C_j R_j\bx-C_j R_j\bar{\bx }} + \sum_{j=1}^p\scal{ R_j\bx  - R_j\bar{\bx }}{C_j R_j\bx - C_j R_j\bar{\bx }} \notag \\
&\geq -\frac{1}{4}\sum_{j=1}^p \ell_j\|(P^\top_j -R_j)(\bx -\bar{\bx })\|^2 +\frac{1}{4}\sum_{j=1}^p \ell_j\left\|(P^\top_j-R_j)(\bx  -\bar{\bx }) +\frac{2}{\ell_j}(C_j R_j\bx -C_j R_j\bar{\bx})\right\|^2 \notag \\ 
&\qquad - \sum_{j=1}^p \frac{1}{\ell_j}\|C_j R_j\bx - C_j R_j\bar{\bx}\|^2 + \sum_{j=1}^p \frac{1}{\ell_j}\|C_j R_j\bx - C_j R_j\bar{\bx}\|^2 \notag \\
& = -\frac{1}{4}\|\diag(\sqrt{\ell})(P^\top -R)(\bx -\bar{\bx})\|^2 \notag \\ 
&\qquad +\frac{1}{4}\left\|(P^\top  -R)(\bx  -\bar{\bx}) +2\diag(\ell)^{-1}(\bC R\bx -\bC R\bar{\bx})\right\|^2_{\diag(\ell)},
\end{align*}
where $P^\top_j$ and $R_j$ denote the $j$-th rows of the matrices $P^\top$ and $R$, respectively. Noting that
\begin{align*}
\|\diag(\sqrt{\ell})(P^\top -R)(\bx -\bar{\bx})\|^2 &=\scal{\bx -\bar{\bx}}{(P -R^\top)\diag(\ell)(P^\top -R)(\bx -\bar{\bx})}, 
\end{align*}
we obtain
\begin{multline*}
\scal{\bx-\bar{\bx}}{\Phi\bx - \Phi(\bar{\bx})} \geq -\frac{1}{2}\scal{\bx -\bar{\bx}}{\Upsilon(\bx -\bar{\bx})} \\ +\frac{1}{4}\left\|(P^\top  -R)(\bx  -\bar{\bx }) +2\diag(\ell)^{-1}(\bC R\bx -\bC R\bar{\bx})\right\|^2_{\diag(\ell)}, 
\end{multline*}
which implies \eqref{ineq:averaged} due to Lemma~\ref{l:GG_ineq}. Since $\alpha \in [0, 1)$, the comonotonicity of $\Gamma$ follows from \eqref{ineq:averaged}. This together with \cite[Proposition~3.3]{BDP22} implies the conical averagedness of $T =\Id -\Gamma$.
\end{proof}

We now arrive at convergence properties of Algorithm~\ref{algo:full}.

\begin{theorem}[Convergence properties]
\label{t:cvg}
Suppose $\zer(\sum_{i=1}^n A_i + \sum_{k=1}^r L_k^* B_k L_k + \sum_{j=1}^p C_j) \neq \varnothing$ and Assumptions~\ref{a:stand} and \ref{a:PSD} hold. Let \( (\bz ^t,\bw ^t)_{t\in\mathbb{N}}\) and \( (\bx^t,\by ^t)_{t\in\mathbb{N}}\) be the sequences generated by Algorithm~\ref{algo:full} with $(\lambda_t)_{t\in\mathbb{N}}$ in $[0,1-\alpha]$.
\begin{enumerate}
\item\label{t:cvg_wco}
If $\liminf_{t\rightarrow +\infty} \lambda_t(1-\frac{\lambda_k}{1-\alpha}) > 0$, then, as $t\to +\infty$,
\begin{enumerate}
\item\label{t:cvg_wco_zTz} 
$(\Id-T)(\bz ^t,\bw ^t)\to 0$ with $\|\frac{1}{t+1}\sum_{i=0}^t (\Id -T)(\bz^i, \bw^i)\|_{\star} = O(1/\sqrt{t})$.
\item\label{t:cvg_wco_zx}
$(\bz^t,\bw^t)\rightharpoonup (\bar{\bz}, \bar{\bw})\in \Fix T$, $(\bx^t,\by ^t) \rightharpoonup (\bar{\bx }, \bar{\by })$, and $(\DB\bL K\bx^t-\bw^t)\rightharpoonup (\DB\bL K\bar{\bx }-\bar{\bw })$ with $\bar{\bx} =\bOne\bar{x}$, $\bar{\by} =\bL (\bOne\bar{x})$, and $(\bar{x},\DB\bL K\bar{\bx}-\bar{\bw})\in \bZ$, provided that the first rows $N$, $P$, $Q$, and $H$ are entirely zeros. 
\end{enumerate}

\item\label{t:cvg_co}
If $C_1,\dots,C_p$ are cocoercive, $Q =0$, and $\sum_{t=0}^{+\infty} \lambda_t(1-\frac{\lambda_t}{1-\alpha}) = +\infty$, then, as $t\to +\infty$, 
\begin{enumerate}
\item\label{t:cvg_co_zTz} 
$(\Id-T)(\bz ^t,\bw ^t)\to 0$ and if $\liminf_{t\rightarrow+\infty} \lambda_t(1-\frac{2\lambda_t}{2-\gamma\tau}) > 0$, then $\|(\Id-T)(\bz^t, \bw^t)\|_{\star} =   o(1/\sqrt{t})$.
\item\label{t:cvg_co_z}
\( (\bz ^t,\bw ^t) \rightharpoonup (\bar{\bz },\bar{\bw }) \in \Fix T \). 
\item\label{t:cvg_co_Bx}
$\bC R\bx^t \rightarrow \bC R\bar{\bx}$ with $\bar{\bx} = J_{\gamma D^{-1}\bA}\Big(D^{-1}(M\bar{\bz} +N\bar{\bx} - \gamma\Phi\bar{\bx} -\gamma H\bL^*(\DB\bL K\bar{\bx} -\bar{\bw}))\Big)$.
\item\label{t:cvg_co_x}
$(\bx^t,\by ^t) \rightharpoonup (\bar{\bx }, \bar{\by })$ and $(\DB\bL K\bx^t-\bw^t)\rightharpoonup (\DB\bL K\bar{\bx}-\bar{\bw })$ with $\bar{\bx} =\bOne\bar{x}$, $\bar{\by} =\bL(\bOne\bar{x})$, and $(\bar{x},\DB\bL K\bar{\bx}-\bar{\bw})\in \bZ$, provided that the first rows of $N$, $P$, and $H$ are entirely zeros.
\end{enumerate}
\end{enumerate}
\end{theorem}
\begin{proof}
As $\zer(\sum_{i=1}^n A_i + \sum_{k=1}^r L_k^* B_k L_k + \sum_{j=1}^p C_j) \neq \varnothing$, it follows from \eqref{e:sol_sets} and Lemma~\ref{l:fixzer} that $\Fix T \ne \varnothing$.

\ref{t:cvg_wco}: By Lemma \ref{l:quasiaveraged} and \cite[Proposition~2.2]{DTT26}, we obtain \ref{t:cvg_wco_zTz} and the Fej\'er monotonicity of $(\bz^t, \bw^t)_{t\in\mathbb{N}}$ with respect to $\Fix T$. Together with Lemma~\ref{l:wcvg}\ref{l:wcvg_Fejer}--\ref{l:wcvg_shadow}, this implies \ref{t:cvg_wco_zx}.

\ref{t:cvg_co}: First, \ref{t:cvg_co_zTz} and \ref{t:cvg_co_z} follow from Lemma \ref{l:averaged} and \cite[Proposition~2.9]{BDP22}. Next, since $(\bar{\bz },\bar{\bw }) \in \Fix T =\zer \Gamma$ and $P^\top - R = UM^\top$ for some $U\in \mathbb{R}^{p\times m}$ (see Remark~\ref{r:OnAssump}\ref{r:OnAssump_UVX}), we derive from Lemma~\ref{l:averaged} and Assumption~\ref{a:PSD} that, for all $t\in \mathbb{N}$,
\begin{multline*}
\scal{\Gamma(\bz^t,\bw^t)}{(\bz^t-\bar\bz^t,\bw^t-\bar\bw^t)}_{\star} \geq \frac{1-\alpha}{2}\|\Gamma(\bz^t,\bw^t)\|_\star^2\\
+\frac{\gamma}{4}\left\|U(\bz^t-T_1(\bz^t,\bw^t))+ 2\diag(\ell)^{-1}(\bC R\bx^t -\bC R\bar{\bx}^t)\right\|^2_{\diag(\ell)}.
\end{multline*}
Letting $t\to +\infty$ and recalling that $\Gamma(\bz^t, \bw^t) =(\Id -T)(\bz ^t,\bw ^t)\to 0$ and $(\bz ^t,\bw ^t)_{t\in \mathbb{N}}$ is bounded, we obtain \ref{t:cvg_co_Bx}. Finally, \ref{t:cvg_co_x} follows from \ref{t:cvg_co_z} and Lemma~\ref{l:wcvg}\ref{l:wcvg_shadow}. 
\end{proof}

We consider a special case of problem \eqref{primal_prob} when $p=0$, that is,
\begin{align}
    \text{find } x \in \mathcal{H} \text{ such that } 0 \in \sum_{i=1}^n A_i x + \sum_{k=1}^r L_k^* B_k L_k x.
\end{align}
In this case, Algorithm~\ref{algo:full} becomes
\begin{align}
    \left[
    \begin{array}{c}
        \bx^t\\
        \by^t
    \end{array}\right]
    &=
    \left[
    \begin{aligned}
        &J_{\gamma D^{-1}\bA }\left( D^{-1} \left(M\bz^t  + N\bx^t -\gamma H\bL^*(\DB\bL K\bx^t -\bw^t )\right)\right)\\
        &J_{\DB^{-1}\bB }\Big(\bL K\bx^t -\DB^{-1}\bw^t  + \bL H^\top \bx^t \Big)
    \end{aligned}\right],\\
    \left[
    \begin{array}{c}
        \bz^{t+1}\\
        \bw^{t+1}
    \end{array}\right]
    &=\left[
    \begin{array}{c}
        \bz^t \\
        \bw^t 
    \end{array}\right] -
    \lambda_k
    \left[
    \begin{array}{c}
        M^\top\bx^t\\
        \DB(\bL H^\top\bx^t - \by^t)
    \end{array}\right].
\end{align}
The following result is a direct consequence of Theorem~\ref{t:cvg}.
\begin{corollary}[Convergence without single-valued operators]
Suppose $\zer(\sum_{i=1}^n A_i + \sum_{k=1}^r L_k^* B_k L_k ) \neq \varnothing$, Assumption~\ref{a:stand} holds, and $\Omega + \alpha MM^\top - \frac{\gamma}{1+\alpha}\Psi \succeq 0$ for some $\alpha \in [0, 1)$. Let \( (\bz ^t,\bw ^t)_{t\in\mathbb{N}}\) and \( (\bx^t,\by ^t)_{t\in\mathbb{N}}\) be the sequences generated by Algorithm~\ref{algo:full} with $(\lambda_t)_{t\in\mathbb{N}}$ in $[0,1-\alpha]$ satisfying $\sum_{t=0}^{+\infty} \lambda_t(1-\frac{\lambda_t}{1-\alpha}) = +\infty$. Then, as $t\to +\infty$, 
\begin{enumerate}
\item
$(\Id-T)(\bz ^t,\bw ^t)\to 0$. Moreover, $\|(\Id-T)(\bz^t, \bw^t)\|_{\star} = o(1/\sqrt{t})$ if $\liminf_{t\rightarrow +\infty} \lambda_t(1-\frac{\lambda_k}{1-\alpha})> 0$.
\item
$(\bz^t,\bw^t)\rightharpoonup (\bar{\bz}, \bar{\bw})\in \Fix T$.
\item 
$(\bx^t,\by ^t) \rightharpoonup (\bar{\bx }, \bar{\by })$ and $(\DB\bL K\bx^t-\bw^t)\rightharpoonup (\DB\bL K\bar{\bx }-\bar{\bw })$ with $\bar{\bx} =\bOne\bar{x}$, $\bar{\by} =\bL (\bOne\bar{x})$, and $(\bar{x},\DB\bL K\bar{\bx}-\bar{\bw})\in \bZ$, provided that the first rows $N$ and $H$ are entirely zeros.
\end{enumerate}
\begin{proof}
This follows directly from Theorem~\ref{t:cvg}\ref{t:cvg_co} by taking $\mathbf{C}=0$, which implies $\Upsilon=0$.
\end{proof}
\end{corollary}

\begin{remark}[On the parameters $\gamma$, $\eta_k$, and $\lambda_t$]
\label{r:PSD} 
Based on Assumption~\ref{a:PSD}, we can derive the range for the parameters involved.
\begin{enumerate}
\item 
It is clear that if Assumption~\ref{a:PSD} is satisfied with $\alpha \geq 0$, then it is also satisfied with $\alpha' > \alpha$, which may yields a larger choice of $\gamma$. However, as shown in Theorem~\ref{t:cvg}, choosing a larger $\alpha$ results in a smaller admissible range for $\lambda_t\in [0, 1-\alpha]$.

\item \label{r:PSD_AB}
Assumption~\ref{a:PSD} holds if and only if 
\begin{align*}
\gamma \leq \lambda_{\min}\left(\Omega+\alpha MM^\top, \frac{1}{1+\alpha}\Psi + \Upsilon\right),   
\end{align*}
where $\lambda_{\min}(X ,Y) :=\inf_{x\notin \ker Y} \frac{\scal{x}{Xx}}{\scal{x}{Yx}}$ denotes the infimum of the generalized Rayleigh quotient. We note that the resolvent parameter of $\bA$ is $\gamma D^{-1}$, rather than $\gamma$, and each $A_i$ is thus associated with the resolvent parameter $\frac{\gamma}{\delta_i}$. Similarly, the resolvent parameter of $\bB$ is $\DB^{-1}$, so each $B_k$ corresponds to the parameter $\frac{1}{\eta_k}$.

\item\label{r:PSD_gammaeta} 
In view of Remark~\ref{r:OnAssump}\ref{r:OnAssump_QK}--\ref{r:OnAssump_UVX}, there exist $U\in \mathbb{R}^{p\times m}$ and $X\in\mathbb{R}^{r\times m}$ such that $P^\top - R = UM^\top$ and $H^\top - K = XM^\top$, and if $Q\neq 0$, then there also exists $V\in \mathbb{R}^{p\times m}$ satisfying $P^\top - Q^\top = VM^\top$. Thus, $\Psi =MX^\top \bL^*\DB \bL XM^\top$ and $\Upsilon = M\Theta M^\top$, where
\begin{align}
\label{def:Theta}
\Theta  := 
\begin{cases}
\frac{1}{2}U^\top\diag(\ell)U &\text{if~} Q=0,\\
U^\top\diag(\ell)U + V^\top\diag(\ell)V &\text{otherwise}.
\end{cases}
\end{align}
When $\Omega = \kappa MM^\top$ for some $\kappa \in [0,+\infty)$, Assumption~\ref{a:PSD} can be expressed in a more explicit form as $M((\kappa +~\alpha)\Id -\frac{\gamma}{1+\alpha} X^\top \bL^*\DB \bL X -\gamma \Theta )M^\top \succeq 0$, which holds whenever 
\begin{align*}
\frac{\gamma}{1+\alpha}\lambda_{\max}(X^\top X)\max_{1\leq k \leq r} \eta_k \|L_k\|^2 +\gamma\lambda_{\max}(\Theta ) \leq \kappa + \alpha.
\end{align*}
\end{enumerate}
\end{remark}

\section{Realizations of Algorithm~\ref{algo:full}}
\label{s:realizations}

\begin{example}[Algorithm for $n=1$, $r=1$, and $p=1$]
Consider a special case of problem~\eqref{primal_prob} with $n=1$, $r=1$, and $p=1$, namely,
\begin{align*}
\text{find } x \in \mathcal{H} \text{ such that } 0 \in A x + L^* B L x + Cx.
\end{align*}
To solve this problem, we apply Algorithm~\ref{algo:full} with $n=2$, $r=1$, $p=1$, and setting $A_1 = 0$, $A_2 = A$. The coefficient matrices are then given by $Q=0$,
\begin{align*}
M = \begin{bmatrix}
1\\
-1
\end{bmatrix}, \,
N = \begin{bmatrix}
0 & 0\\
2 & 0
\end{bmatrix}, \,
D = \begin{bmatrix}
1 & 0\\
0 & 1
\end{bmatrix}, \,
\DB = \begin{bmatrix}
\eta\\
\end{bmatrix}, \,
H = \begin{bmatrix}
0\\
1
\end{bmatrix}, \, 
K = \begin{bmatrix}
1 & 0
\end{bmatrix},
P = \begin{bmatrix}
0\\
1
\end{bmatrix}, \, 
R = \begin{bmatrix}
1 & 0
\end{bmatrix}.
\end{align*}
Algorithm~\ref{algo:full} now reduces to
\begin{align}
\label{eq:xy}
&\begin{cases}
x^t &= J_{\gamma A}(z^t - \gamma Cz^t - \gamma L^* (\eta Lz^t - w^t))\\
y^t &= J_{\frac{1}{\eta} B}\left(Lz^t - \frac{1}{\eta} w^t + Lx^t\right),
\end{cases} \\
\label{eq:zw}
&\begin{cases}
z^{t+1} = z^t - \lambda_t(z^t - x^t)\\
w^{t+1} = w^t - \lambda_t\eta (Lx^t - y^t).
\end{cases}
\end{align}
Using Moreau decomposition \cite[Theorem 14.3(ii)]{BC17}, $y^t$ in \eqref{eq:xy} becomes
\begin{align*}
y^t = Lz^t - \frac{1}{\eta} w^t + Lx^t - \frac{1}{\eta} J_{\eta B^{-1}}\left(\eta Lz^t - w^t + \eta Lx^t\right).
\end{align*}
Set $u^t :=\eta Lz^t -w^t$ and $v^t := u^t +\eta Lx^t -\eta y^t =\eta Lz^t -w^t +\eta Lx^t -\eta y^t$. Then \eqref{eq:xy} becomes 
\begin{align*}
x^t =J_{\gamma A}(z^t - Cz^t -\gamma L^*u^t) \text{~~and~~} v^t =J_{\eta B^{-1}}(u^t +\eta Lx^t).    
\end{align*}
Since $w^t =\eta Lz^t -u^t$ and $\eta (Lx^t -y^t) =v^t -u^t$, the second equation in \eqref{eq:zw} becomes $\eta Lz^{t+1} -u^{t+1} =\eta Lz^t -u^t -\lambda_t (v^t -u^t)$, or equivalently,
\begin{align*}
u^{t+1} =u^t +\lambda_t (v^t -u^t) +\eta(Lz^{t+1} -Lz^t).
\end{align*}
We obtain the algorithm
\begin{align*}
&\begin{cases}
x^t &= J_{\gamma A}(z^t - Cz^t - \gamma L^*u^t)\\
v^t &= J_{\eta B^{-1}}(u^t + \eta Lx^t),
\end{cases} \\
&\begin{cases}
z^{t+1} = z^t - \lambda_t(z^t - x^t)\\
u^{t+1} = u^t + \lambda_t (v^t - u^t) + \eta(Lz^{t+1} -Lz^t).
\end{cases}
\end{align*}
In turn, when $\lambda_t = 1$, this algorithm simplifies to 
\begin{align}
\label{algo:similar_ConVu}
\begin{cases}
z^{t+1} = J_{\gamma A}(z^t - Cz^t - \gamma L^* u^t)\\
u^{t+1} = J_{\eta B^{-1}}(u^t + \eta Lz^{t+1}) + \eta(Lz^{t+1} - Lz^t),
\end{cases}
\end{align}
which closely resembles, but is not exactly the same as, the Condat--V\~u algorithm \cite{Con13, Vu13}, where the dual update takes the form $u^{t+1} = J_{\eta B^{-1}}(u^t + \eta L(2z^{t+1} - z^t))$. 

It is straightforward to verify that Assumption~\ref{a:stand} holds. Moreover, $\Omega =MM^\top$ and $H^\top - K = X M^\top, P^\top - R = UM^\top$ with $X = U= [-1]$. Hence, by Remark~\ref{r:PSD}\ref{r:PSD_gammaeta}, Assumption~\ref{a:PSD} is satisfied whenever $\alpha=0$ (so that $\lambda_t =1$) and $\gamma\eta\|L\|^2 + \frac{\gamma\ell}{2} \leq 1$. 
\end{example}

\begin{example}[Algorithm for $r =0$]
We consider another special case of problem \eqref{algo:full} when $r=0$, that is,
\begin{align*}
    \text{find } x \in \mathcal{H} \text{ such that } 0 \in \sum_{i=1}^n A_i x + \sum_{j=1}^p C_jx.
\end{align*}
With $H=0$ and $K=0$, Algorithm~\ref{algo:full} becomes
\begin{align}
\label{algo:DTT26}
\begin{cases}
\bx^t &= J_{\gamma D^{-1}\bA}\left(D^{-1} \left(M\bz^t + N\bx^t - \gamma ((P-Q)\bC R\bx^t + Q\bC P^\top\bx^t)\right)\right) \\
\bz^{t+1} &= \bz^t - \lambda_t M^\top\bx^t,
\end{cases}
\end{align}
which not only recovers but also improves \cite[Algorithm 1]{DTT26} significantly. When $Q=0$ and each $C_j$ is cocoercive, \eqref{algo:DTT26} reduces to \cite[Algorithm 1]{ACGN25}. The scheme \eqref{algo:DTT26} further extends many existing algorithms including the frugal and decentralized splittings \cite{Tam23}, the \emph{forward-backward algorithms devised by graphs} \cite{ACL24}, the forward-backward and forward-reflected-backward algorithms for ring networks \cite{AMTT23}, the forward-reflected-backward algorithm \cite{MT20}, the sequential and parallel forward-Douglas--Rachford algorithms \cite{BCLN22}, the \emph{generalized forward-backward algorithm} \cite{RFP13}. Moreover, the scheme generates different product-space formulations of the Davis--Yin algorithm, including a reduced dimensional variant, and a generalization of Ryu splitting algorithm; see \cite[Section~4]{DTT26}.

In this case, Assumptions~\ref{a:PSD} reduces to 
\begin{align}
\label{a:PSD_old}
\Omega +\alpha MM^\top -\gamma \Upsilon =\Omega +M(\alpha \Id -\gamma \Theta )M^\top \succeq 0 \text{~for some~} \alpha \in [0,1),
\end{align}
where $\Theta $ is given by \eqref{def:Theta}. The introduction of $\alpha$ yields significantly weaker assumption than \cite[Assumption 3.2(v)]{DTT26} and helps unifies the key inequalities and its variants in \cite[Lemma 3.6 and Remark 3.7]{DTT26}, thereby allowing a larger explicit stepsize range. To be more specific, one sufficient condition for \eqref{a:PSD_old} is 
\begin{align*}
\Omega \succeq 0 \text{~~and~~} \gamma \leq \frac{\alpha}{\lambda_{\max}(\Theta )} \text{~for some~} \alpha \in (0,1),    
\end{align*}
which includes the case when
\begin{align*}
\Omega \succeq 0,\ \gamma < \frac{1}{\lambda_{\max}(\Theta )}, \text{~~and~~} \alpha =\lambda_{\max}(\Theta ).    
\end{align*}
When $\ell_1=...=\ell_p = \ell$, the latter upper bound for $\gamma$ becomes $\frac{2}{\ell\lambda_{\max}(U^\top U)} =\frac{2}{\ell\|U\|^2}$ if $Q =0$ (matching \cite[Theorem~3.10]{DTT26}), and $\frac{1}{\ell\lambda_{\max}(U^\top U +V^\top V)}$ if $Q \neq 0$, which improves on $\frac{1}{\ell(\|U\|^2 +\|V\|^2)}$ from \cite[Theorem~3.9]{DTT26}.

Another sufficient condition for \eqref{a:PSD_old} is $\Omega -\gamma\Upsilon \succeq 0$, which reduces to \cite[Equations~(18) and (20)]{DTT26} if $\ell_1=...=\ell_p = \ell$, and to \cite[Assumption 4.7]{ACGN25} if $Q =0$ and $\gamma =1$.

When $\Omega = \kappa MM^\top$ for some $\kappa \in [0, +\infty)$, as in Remark~\ref{r:PSD}\ref{r:PSD_gammaeta}, \eqref{a:PSD_old} yields an even larger upper bound for $\gamma$, namely
\begin{align}
    \gamma \leq \frac{\kappa + \alpha}{\lambda_{\max}(\Theta )}.
\end{align}
This situation commonly arises in graph-based splitting algorithms when the underlying graphs coincide. In that case, the condition $\Omega = \kappa MM^\top$ is automatically satisfied; see Example~\ref{ex:E'=E}. 

Moreover, the explicitness condition in Remark~\ref{r:OnAlgo} is simple, clear, and easy to customize. The method also accommodates different cocoercive or Lipschitz constants for the operator $C_j$ and the convergence analysis in Theorem~\ref{t:cvg} is unified for both cases with or without cocoercivity through the introduction of the notion of \emph{quasicomonotonicity}.
\end{example}

In the sequel, we generate a class of algorithms with graph-based structures which easily satisfy Assumption~\ref{a:stand}. Moreover, we provide a simple choice for the coefficient matrices that achieves a large range for the stepsizes. Here, we briefly introduce the main idea of graph theory and refer readers to \cite[Section 4]{DTT26} for more detail. 

Let $G =(\mathcal{V},\mathcal{E})$, where $\mathcal{V}=\{v_1,\dots,v_n\}$, $\mathcal{E}=\{e_1,\dots,e_q\}\subseteq \mathcal{V}\times \mathcal{V}$, be a weighted undirected connected graph with $|\mathcal{V}|=n \geq 2$, $|\mathcal{E}|=q \geq n-1$, and weight matrix $W=(w_{ij})\in [0, +\infty)^{n\times n}$. Let $G'=(\mathcal{V}, \mathcal{E'})$ be a weighted connected subgraph of $G$ with $\mathcal{E'}\subseteq \mathcal{E}$, and weight matrix $W'=(w_{ij}')\in[0, +\infty)^{n\times n}$ such that for all $(v_i,v_j)\in \mathcal{E'}$,
\begin{align*}
w_{ij}' \leq w_{ij}.
\end{align*}
i.e., the weight on each edge of $G'$ is not greater than the weight on the corresponding edge of $G$. 

Let $N=(N_{ij})\in\mathbb{R}^{n\times n}$ where $N_{ij} = w_{ji}$ if $i>j$ and $N_{ij}=0$ if $i \leq j$. Thus $N$ is a strictly lower triangular matrix. Let $\Deg(G)=\diag(d_1,\dots,d_n)$, where $d_i=\sum_{j=1}^n w_{ij}$, is the \emph{degree matrix} of $G$. Let $D=\diag(\delta_1,\dots,\delta_n)$, where $\delta_i =\frac{1}{2}d_i$. It is clear that $2D-N-N^\top = \Deg(G)- W = \Lap(G)$. Furthermore, let $M = (M_{ij}) \in\mathbb{R}^{n\times (n-1)}$ be an \emph{onto decomposition} of the \emph{Laplacian matrix} $\Lap(G') =\Deg(G') -W'$, i.e., $MM^\top = \Lap(G')$. If $G'$ is a tree, then $M$ can be chosen as the \emph{incidence matrix} $\Inc(G'^{\sigma})$ for some orientation $\sigma$ of $G'$ (see \cite[Remark 2.17]{ACL24}). This graph-based selection of $M$, $N$, and $D$ satisfies Assumption~\ref{a:stand}\ref{a:stand_kerM}--\ref{a:stand_N}.

\begin{example}[Algorithms based on ring and sequential graphs]
\label{eg:ring_seq}
Let $G$ be a ring graph and $G'$ be a sequential graph, both with unit weights. Since $G'$ is a tree, we take $M$ as the incidence matrix. Then $D = \diag(1,\dots,1)$, $\DB =\diag(\eta_1,\dots,\eta_r)$,
\begin{alignat*}{2}
    M_{ij} &= 
    \begin{cases}
        1 &\text{~if } i\in\{1,\dots, n-1\},\ j=i\\
        -1 &\text{~if } i\in\{2, \dots,n\},\ j=i-1\\
        0 &\text{~otherwise},
    \end{cases} \qquad\qquad &
    N_{ij} &= 
    \begin{cases}
        1 &\text{~if } i\in\{2,\dots,n\},\ j=i-1\\
        1 &\text{~if } i=n,\ j=1\\
        0 &\text{~otherwise},
    \end{cases}\\
    H &=\left[ 
        \begin{array}{c} 
          0_{(n-1)\times r} \\ 
          \hline 
          1_{1 \times r} 
        \end{array} 
        \right], 
    &
    K &=\left[
        \begin{array}{c|c}
            1_{r\times 1} & 0_{r\times (n-1)}
        \end{array}\right].
\end{alignat*}
It can be seen that Assumption~\ref{a:stand}\ref{a:stand_H} holds and $\Omega =M\bOne \bOne^\top M^\top$. Taking $X=-1_{r\times (n-1)}$, we have $H^\top-K =XM^\top$ and $\Psi =MX^\top \bL^* \DB\bL XM^\top =M(\sum_{k=1}^r \eta_k L_k^* L_k)\bOne \bOne^\top M^\top$. According to Remark~\ref{r:PSD}\ref{r:PSD_gammaeta}, once Assumption~\ref{a:stand}\ref{a:stand_PR} also holds, Assumption~\ref{a:PSD} becomes $M\left(\left(\Id - \frac{\gamma}{1+\alpha}\sum_{k=1}^r \eta_k L_k^* L_k\right) \bOne \bOne^\top + (\alpha\Id - \gamma \Theta ) \right)M^\top \succeq 0$, which is satisfied as soon as
\begin{align}
\label{eg:4.3}
    (n-1)\left(1 -\frac{\gamma}{1+\alpha} \sum_{k=1}^r \eta_k\|L_k\|^2\right) +\alpha -\gamma \lambda_{\max}(\Theta ) \geq 0,
\end{align}
where $\Theta $ is defined in \eqref{def:Theta}.

\emph{Case 1:} Each $C_j$ is $\frac{1}{\ell_j}$-cocoercive. Then we set
\begin{align*}
P =\left[ 
    \begin{array}{c} 
      0_{(n-1)\times p} \\ 
      \hline 
      1_{1 \times p} 
    \end{array} 
    \right],\ 
R =\left[
    \begin{array}{c|c}
        1_{p\times 1} & 0_{p\times (n-1)}
    \end{array}\right],
\text{~~and~~} Q =0,
\end{align*}
which satisfy Assumption~\ref{a:stand}\ref{a:stand_PR}. Algorithm~\ref{algo:full} becomes
\begin{align}
\begin{aligned}
    &\begin{cases}
        x_1^t &= J_{\gamma A_1}(z_1^t)\\
        x_i^t &= J_{\gamma A_i}(z_i^t - z_{i-1}^t + x_{i-1}^t),
        \quad i\in\{2,\dots,n-1\}\\
        x_n^t &= J_{\gamma A_n}(- z_{n-1}^t + x_1^t + x_{n-1}^t - \gamma\sum_{j=1}^p C_j x_1^t - \gamma\sum_{j=1}^r L_j^*(\eta_j L_j x_1^t - w_j^t))\\
        y_k^t &= J_{\frac{1}{\eta_k} B_k}(L_k x_1^t -\frac{1}{\eta_k} w_k^t + L_k x_n^t), \quad k\in\{1,\dots, r\},
    \end{cases}\\
    &\begin{cases}
        z_i^{t+1} &= z_i^t - \lambda_t(x_i^t - x_{i+1}^t), \quad i\in\{1,\dots,n-1\}\\
        w_k^{t+1} &= w_k^t - \lambda_t\eta(L_k x_n^t - y_k^t), \quad k\in\{1,\dots, r\},
    \end{cases}
\end{aligned}
\end{align}
which reduces to \cite[Algorithm~1]{ABT23} when $p=0$ (no single-valued operator $C_j)$ and $\gamma=1$. 

Letting $U = -1_{p\times (n-1)}$, we obtain $P^\top - R = UM^\top$ and $\lambda_{\max}(\Theta )=\lambda_{\max}(\frac{1}{2}U^\top\diag(\ell)U) = \frac{n-1}{2}\sum_{j=1}^p \ell_j$. The condition \eqref{eg:4.3} is equivalent to 
\begin{align*}
\frac{\gamma}{2}\sum_{j=1}^p \ell_j +\frac{\gamma}{1+\alpha}\sum_{k=1}^r \eta_k\|L_k\|^2 \leq \frac{n-1+\alpha}{n-1}.
\end{align*}
In particular, when $p=0$, the condition reduces to $\gamma\sum_{k=1}^r \eta_k\|L_k\|^2 \leq \frac{(1+\alpha)(n-1+\alpha)}{n-1}$.

\emph{Case 2:} Each $C_j$ is monotone and $\ell_j$-Lipschitz continuous. Then, to satisfy Assumption~\ref{a:stand}\ref{a:stand_PR}, we choose
\begin{align*}
P =\left[ 
    \begin{array}{c} 
      0_{(n-2)\times p} \\ 
      \hline 
      1_{1 \times p}\\
      \hline
      0_{1 \times p}
    \end{array} 
    \right],\
    R =\left[
    \begin{array}{c|c}
        1_{p\times 1} & 0_{p\times (n-1)}
    \end{array}\right],
    \text{~~and~~}
    Q =\left[ 
        \begin{array}{c} 
          0_{(n-1)\times p}\\  
          \hline 
          1_{1 \times p} 
        \end{array} 
    \right].
\end{align*}
Algorithm~\ref{algo:full} now takes the form
\begin{align}
\begin{aligned}
    &\begin{cases}
        x_1^t &= J_{\gamma A_1}(z_1^t)\\
        x_i^t &= J_{\gamma A_i}(z_i^t - z_{i-1}^t + x_{i-1}^t),
        \quad i\in\{2,\dots,n-2\}\\
        x_{n-1}^t\!\!\! &= J_{\gamma A_{n-1}}(z_{n-1}^t - z_{n-2}^t + x_{n-2}^t - \gamma\sum_{j=1}^p C_j x_1^t)\\
        x_n^t &= J_{\gamma A_n}(- z_{n-1}^t + x_1^t + x_{n-1}^t + \gamma\sum_{j=1}^p C_j (x_1^t -x_{n-1}^t)  - \gamma\sum_{j=1}^r L_j^*(\eta_j L_j x_1^t - w_j^t))\\
        y_k^t &= J_{\frac{1}{\eta_k} B_k}(L_k x_1^t -\frac{1}{\eta_k} w_k^t + L_k x_n^t), \quad k\in\{1,\dots, r\},
    \end{cases}\\
    &\begin{cases}
        z_i^{t+1} &= z_i^t - \lambda_t(x_i^t - x_{i+1}^t), \quad i\in\{1,\dots,n-1\}\\
        w_k^{t+1} &= w_k^t - \lambda_t\eta_k(L_k x_n^t - y_k^t), \quad k\in\{1,\dots, r\}.
    \end{cases}
\end{aligned}
\end{align}
By definition, $P^\top - R = UM^\top$ and $P^\top-Q^\top = VM^\top$ with $U = [-1_{p\times (n-2)} \ | \ 0_{p\times 1}]$ and $V = [0_{p\times (n-2)} \ | \ 1_{p\times 1}]$. We see that $\lambda_{\max}(\Theta )=\lambda_{\max}(V^\top\diag(\ell)V + U^\top\diag(\ell)U) = (n-2)\sum_{j=1}^p \ell_j$. The condition \eqref{eg:4.3} is satisfied if and only if
\begin{align}
    \frac{\gamma(n-2)}{(n-1)}\sum_{j=1}^p \ell_j +\frac{\gamma}{1+\alpha}\sum_{k=1}^r \eta_k\|L_k\|^2 \leq \frac{n-1+\alpha}{n-1}.
\end{align}
\end{example}

From now on, we work in the setting where each $C_j$ is $\frac{1}{\ell_j}$-cocoercive, which allows us to take $Q =0$. In the absence of cocoercivity, one may instead choose any $Q\in \mathbb{R}^{n \times p}$ satisfying $Q^\top\bOne =\bOne$ together with conditions \eqref{a:PQR_lower_1b}--\eqref{a:PQR_lower_1c} (if an explicit form is desired).

\begin{example}[Algorithms based on complete, sequential, and star graphs]
\label{ex:E'=E}
We consider the particular instance of problem \eqref{primal_prob} when $r =p =n -1$. Let $\mathcal{E'} =\mathcal{E}$ with $w_{ij}=\kappa+1$, $\kappa\in [0, +\infty)$, and $w'_{ij}=1$ for $\{v_i, v_j\}\in\mathcal{E}$, meaning that $G$ and $G'$ have the same graph structure but may differ in their edge weights. Then $\Omega = 2D-N-N^\top - MM^\top = \Lap(G)-\Lap(G')=\kappa \Lap(G') = \kappa MM^\top$.

\emph{Case 1:} $G$ and $G'$ are complete graphs. In this case, we take $M$ as in \cite[Proposition A.2]{ACL24}, which is not the incidence matrix of $G'$ but still satisfies $MM^\top = \Lap(G')$. We now derive an algorithm based on the structure of complete graphs. 

\paragraph{Complete graph algorithm:} 
\begin{alignat*}{3}
    M_{ij} &= 
        \begin{cases}
            a_i := \sqrt{\frac{(n-i)n}{n-i+1}} &\text{~if~} i=j\\ 
            t_j := -\sqrt{\frac{n}{(n-j)(n-j+1)}} &\text{~if~} i > j \\ 
            0 &\text{~otherwise},
        \end{cases}
    \qquad\qquad &
    N_{ij}&=
    \begin{cases}
        \kappa + 1 &\text{~if~} i > j \\
        0 &\text{~otherwise},
    \end{cases}\\
    D &= \frac{(\kappa+1)(n-1)}{2}\diag(1,\dots,1),
    &
    \DB &= \eta \DB' = \diag(\eta a_1^2, \dots, \eta a_{n-1}^2),\\
    H_{ij}=P_{ij}&=\begin{cases}
            \frac{1}{n-j} &\text{~if~} i > j\\ 
            0 &\text{~otherwise},
        \end{cases}
    &
    K=R&=\left[
        \begin{array}{c|c}
            \Id_{n-1} & 0_{(n-1)\times 1}
        \end{array}\right].
\end{alignat*}
Then Assumption~\ref{a:stand}\ref{a:stand_PR}--\ref{a:stand_H} is satisfied. Taking $U=-\diag(\frac{1}{a_1},\dots,\frac{1}{a_{n-1}}) \in \mathbb{R}^{(n-1)\times (n-1)}$, we have $P^\top -R =UM^\top$ and $\lambda_{\max}(\Theta )=\lambda_{\max}(\frac{1}{2}U^\top\diag(\ell)U) = \frac{1}{2}\max_{1\leq k\leq n-1} \frac{\ell_k}{a_k^2}$. Moreover, $\sqrt{\DB'}(H^\top - K) = -M^\top$, which yields $\Psi = (H-K^\top)\bL^* \eta \DB' \bL(H^\top - K) = M (\eta \bL^* \bL) M^\top$. Assumption~\ref{a:PSD} is guaranteed when $(\kappa + \alpha) - \frac{\gamma\eta}{1+\alpha}\max_{1\leq k\leq n-1}\|L_k\|^2 - \frac{\gamma}{2}\max_{1\leq k\leq n-1}\frac{\ell_k}{a_k^2} \geq 0$, which is equivalent to
\begin{align}
\label{stepsize_complete}
    \gamma < \frac{2(\kappa+\alpha)}{\max_{1\leq k\leq n-1} \frac{\ell_k}{a_k^2}} \text{~~and~~}
    \eta_k =\eta a_k^2 \text{~with~} \eta \leq \frac{(1+\alpha)\left(2(\kappa+\alpha)-\gamma \max_{1\leq k\leq n-1}\frac{\ell_k}{a_k^2} \right)}{2\gamma \max_{1\leq k\leq n-1}\|L_k\|^2}.
\end{align}
Algorithm~\ref{algo:full} reduces to
\begin{align}
    \label{algo:complete}
    \begin{aligned}
    &\begin{cases}
        x_1^t &= J_{\frac{2\gamma}{(\kappa+1)(n-1)}A_1}\Big(\frac{2}{(\kappa+1)(n-1)}a_1 z_1^t\Big)\\
        x_i^t &= J_{\frac{2\gamma}{(\kappa+1)(n-1)}A_i}\Bigg(\frac{2}{(\kappa+1)(n-1)}\Big(a_i z_i^t + \sum_{j=1}^{i-1} t_j z_j^t +\sum_{j=1}^{i-1} (\kappa+1)x_j^t -\gamma \sum_{j=1}^{i-1}\frac{1}{n-j} C_j x_j^t \\ 
        &\qquad\qquad\qquad\qquad\qquad\qquad - \gamma \sum_{j=1}^{i-1} \frac{1}{n-j} L_j^*(\eta a_j^2 L_j x_j^t - w_j^t)\Big)\Bigg),\quad i\in\{2,\dots, n-1\}\\
        x_n^t &= J_{\frac{2\gamma}{(\kappa+1)(n-1)}A_n}\Bigg(\frac{2}{(\kappa+1)(n-1)}\Big(\sum_{j=1}^{n-1}t_j z_j^t + \sum_{j=1}^{n-1} (\kappa+1)x_j^t - \gamma \sum_{j=1}^{n-1} \frac{1}{n-j} C_j x_j^t \\
        &\qquad\qquad\qquad\qquad\qquad\qquad - \gamma \sum_{j=1}^{n-1} \frac{1}{n-j} L_j^*(\eta a_j^2 L_j x_j^t - w_j^t)\Big)\Bigg)\\
        y_k^t &= J_{\frac{1}{\eta a_k^2} B_k} \Big(L_k x_k^t - \frac{1}{\eta a_k^2} w_k^t +  L_k (\frac{1}{n-k}\sum_{j=k+1}^n x_j^t)\Big), \quad k\in\{1,\dots, n-1\},
    \end{cases}\\
    &\begin{cases}
        z_i^{t+1} &= z_i^t - \lambda_t (a_i x_i^t + t_i\sum_{j=i+1}^n x_j^t), \quad i\in\{1,\dots,n-1\}\\
        w_k^{t+1} &= w_k^t - \lambda_t \eta a_k^2 (L_k (\frac{1}{n-k}\sum_{j=k+1}^n x_j^t) - y_k^t), \quad k\in\{1,\dots,n-1\}.
    \end{cases}
    \end{aligned}
\end{align}

\emph{Case 2:} $G'$ is a tree. Then $|\mathcal{E}'| = n-1$. Let $M=\Inc(G'^\sigma)\in\mathbb{R}^{n \times (n-1)}$ be the incidence matrix of $G'$ given by
\begin{align*}
\Inc(G'^\sigma)_{ij}=
\begin{cases}
1 \quad &\text{~if edge $e_j$ leaves node $v_i$}\\
-1 &\text{~if edge $e_j$ enters node $v_i$}\\
0 &\text{~otherwise}.
\end{cases}
\end{align*}
Let $H=P=(P_{ij})\in\mathbb{R}^{n \times (n-1)}$ and $K=R=(R_{ij})\in\mathbb{R}^{(n-1) \times n}$ are defined based on the incidence matrix of $G'$,
\begin{align*}
P_{ij}=
\begin{cases}
1 \quad &\text{~if edge $e_j$ enters node $v_i$}\\
0 &\text{~otherwise},
\end{cases}
\qquad 
R_{ji}=
\begin{cases}
1 \quad &\text{~if edge $e_j$ leaves node $v_i$}\\
0 &\text{~otherwise}.
\end{cases}
\end{align*}
Using the coefficient matrices defined above, Assumption~\ref{a:stand}\ref{a:stand_PR}--\ref{a:stand_H} is readily verified. Letting $U = X =-\Id_{n-1}$, we have that $P^\top -R =UM^\top$, $H-K^\top = XM^\top$, $\lambda_{\max}(\Theta )=\lambda_{\max}(\frac{1}{2}U^\top\diag(\ell)U) = \frac{1}{2}\max_{1\leq k\leq n-1} \ell_k$, and $\lambda_{\max}(X^\top X) = 1$. In view of Remark~\ref{r:PSD}\ref{r:PSD_gammaeta}, Assumption~\ref{a:PSD} holds whenever 
\begin{align}
\label{stepsize_tree}
    \gamma < \frac{2(\kappa+\alpha)}{\max_{1\leq k\leq n-1}\ell_k} \quad \text{~and~}\quad 
    \eta_k \leq \frac{(1+\alpha)\left(2(\kappa+\alpha)-\gamma \max_{1\leq k\leq n-1}\ell_k\right)}{2\gamma \|L_k\|^2}, \quad k\in\{1,\dots,r\}.
\end{align}
To illustrate, we give examples for algorithms based on the sequential and star graphs, both of which have tree structures.

\paragraph{Sequential graph algorithm:}
\begin{alignat*}{3}
    M_{ij} &=
    \begin{cases}
        1 &\text{~if~} i\in\{1,\dots, n-1\}, \ j=i \\
        -1 &\text{~if~}i\in\{2,\dots,n\},\ j=i-1\\ 
        0 &\text{~otherwise},
    \end{cases}
    \qquad &
    N_{ij} &=
    \begin{cases}
        \kappa +1 &\text{~if~}\ i\in\{2,\dots,n\},\ j=i-1 \\
        0 &\text{~otherwise},
    \end{cases}\\
    D &= \frac{\kappa+1}{2}\diag(1,2,\dots,2,1),
    &
    \DB &= \diag(\eta_1,\dots,\eta_{n-1}),\\
    H=P&=\left[ 
        \begin{array}{c} 
          0_{1\times (n-1)} \\ 
          \hline 
          \Id_{n-1} 
        \end{array} 
        \right],
        &
    K=R&=\left[
        \begin{array}{c|c}
            \Id_{n-1} & 0_{(n-1)\times 1}
        \end{array}\right].
\end{alignat*}
Algorithm~\ref{algo:full} then becomes 
\begin{align}
\label{algo:seq}
\begin{aligned}
    &\begin{cases}
        x_1^t &= J_{\frac{2\gamma}{\kappa+1} A_1} \left(\frac{2}{\kappa+1}z_1^t\right)\\
        x_i^t &= J_{\frac{\gamma}{\kappa+1} A_{i}}\Big( \frac{1}{\kappa+1}(z_i^t - z_{i-1}^t + (\kappa+1)x_{i-1}^t - \gamma C_{i-1}x_{i-1}^t - \gamma L_{i-1}^*(\eta_{i-1} L_{i-1}x_{i-1}^t - w_{i-1}^t))\Big),\\ 
        &\qquad \qquad i\in\{2,\dots,n-1\}\\
        x_{n}^t &= J_{\frac{2\gamma}{\kappa+1} A_{n}} \left(\frac{2}{\kappa+1}\left(-z_{n-1}^t + (\kappa+1)x_{n-1}^t - \gamma C_{n-1} x_{n-1}^t - \gamma L_{n-1}^*(\eta_{n-1} L_{n-1}x_{n-1}^t - w_{n-1}^t)\right)\right)\\
        y_k^t &= J_{\frac{1}{\eta_k} B_k}\Big(L_k x_k^t - \frac{1}{\eta_k} w_k^t + L_k x_{k+1}^t\Big), \quad k\in\{1,\dots,n-1\},
    \end{cases}\\
    &\begin{cases}
        z_i^{t+1} &= z_i^t - \lambda_t(x_i^t-x_{i+1}^t), \quad i\in\{1,\dots,n-1\}\\
        w_k^{t+1} &= w_k^t - \lambda_t \eta_k (L_k x_{k+1}^t - y_k^t), \quad k\in\{1,\dots,n-1\}.
    \end{cases}
    \end{aligned}
\end{align}

\paragraph{Star graph algorithm:}
\begin{alignat*}{3}
    M_{ij} &=
    \begin{cases}
        1 &~\text{if }i=1,\ j\in\{1, \dots, n-1\} \\
        -1 &~\text{if }i\in\{2,\dots,n\},\ j=i-1\\ 
        0 &\text{~otherwise},
    \end{cases}
    \qquad\qquad &
    N_{ij} &=
    \begin{cases}
        \kappa + 1 &~\text{if }j=1,\ i\in\{2,\dots,n\} \\
        0 & \text{~otherwise},
    \end{cases}\\
    D &= \frac{\kappa+1}{2}\diag(n-1, 1, \dots,1),
    &
    \DB &= \diag(\eta_1,\dots,\eta_{n-1}),\\
    H=P&=\left[ 
        \begin{array}{c} 
          0_{1\times (n-1)} \\ 
          \hline 
          \Id_{n-1} 
        \end{array} 
        \right], 
    &
    K=R&=\left[
        \begin{array}{c|c}
            1_{(n-1)\times 1} & 0_{(n-1)\times (n-1)}
        \end{array}\right].
\end{alignat*}
Algorithm~\ref{algo:full} now takes the form
\begin{align}
\label{algo:star}
\begin{aligned}
    &\begin{cases}
        x_1^t &= J_{\frac{2\gamma}{(\kappa+1)(n-1)} A_1}\Big(\frac{2}{(\kappa+1)(n-1)}\sum_{j=1}^{n-1} z_j^t\Big)\\
        x_i^t &= J_{\frac{2\gamma}{\kappa+1} A_{i}}\left(\frac{2}{\kappa+1}\Big((\kappa+1)x_1^t - z_{i-1}^t - \gamma C_{i-1}x_1^t - \gamma L_{i-1}^*(\eta_{i-1} L_{i-1}x_1^t - w_{i-1}^t)\Big)\right), \quad i\in\{2,\dots,n\}\\
        y_k^t &= J_{\frac{1}{\eta_k} B_k}\Big(L_k x_1^t - \frac{1}{\eta_k} w_k^t + L_k x_{k+1}^t\Big), \quad k\in\{1,\dots,n-1\},
    \end{cases}\\
    &\begin{cases}
        z_i^{t+1} &= z_i^t - \lambda_t(x_1^t-x_{i+1}^t), \quad i\in\{1,\dots,n-1\}\\
        w_k^{t+1} &= w_k^t - \lambda_t \eta_k (L_k x_{k+1}^t - y_k^t), \quad k\in\{1,\dots,n-1\}.
    \end{cases}
    \end{aligned}
\end{align}
\end{example}

\begin{remark}
In the case where $p=0$, i.e., no single-valued operator $C_j$, the stepsize range allowed in \emph{Case 1} of Example~\ref{eg:ring_seq} is $\gamma\sum_{k=1}^r \eta_k\|L_k\|^2 \leq \frac{(1+\alpha)(n-1+\alpha)}{n-1}$, for $\alpha \in [0,1)$. This range significantly improves upon that in \cite[Corollary~1]{ABT23} which is only for $\gamma=1$, $\eta_1=\dots=\eta_r=\eta$, and $\alpha=0$. 
    
Moreover, we not only extend the result to the setting when $p\neq 0$ but also design several algorithms with different graph structures including the complete graph \eqref{algo:complete}, the sequential graph \eqref{algo:seq}, and the star graph \eqref{algo:star} algorithms. It follows from \eqref{stepsize_complete} and \eqref{stepsize_tree} that when $p=0$ (so that $\ell_1=\dots=\ell_p=0$), the stepsize range for the complete graph algorithm is
\begin{align}
\label{complete_p=0}
\gamma\eta\max_{1\leq k\leq n-1} \|L_k\|^2 \leq (1+\alpha)(\kappa+\alpha),
\end{align}
while for the sequential and star graph algorithms it is
\begin{align}
\label{tree_p=0}
\gamma \max_{1\leq k \leq n-1} \eta_k\|L_k\|^2 \leq (1+\alpha)(\kappa + \alpha).
\end{align}
As noted in Remark~\ref{r:PSD}\ref{r:PSD_AB}, the resolvent parameters associated with $A_i$ and $B_k$ are $\frac{\gamma}{\delta_i}$ and $\frac{1}{\eta_k}$, respectively. It is clear that handling the linear operators $L_1, \dots, L_r$  directly, not rather than via the product space reformulation, can lead to alternative algorithmic designs with significantly larger stepsizes, as in \eqref{complete_p=0} and \eqref{tree_p=0}.
\end{remark}

\begin{remark}
\label{r:n+1}
Although the new algorithms in Example~\ref{ex:E'=E} are derived for solving problem~\eqref{primal_prob} when $r=p=n-1$, they can be readily adapted to the case $n=r=p$. Specifically, we apply the complete, sequential, and star graph algorithms to problem~\eqref{primal_prob} with $(n+1)$ maximally monotone operators $A_1,\dots,A_{n+1}$, $n$ compositions of $B_1,\dots,B_n$ with bounded linear operator $L_1,\dots,L_n$, and $n$ cocoercive or monotone Lipschitz operators $C_1,\dots,C_n$. Then by letting operator $A_1=0$ (or $A_{n+1}=0$), we obtain the corresponding algorithms. This is useful in some cases, as shown in the next section.
\end{remark}

\section{Numerical experiments}
\label{s:num_exper}

In this section, we present a numerical experiment on the decentralized fused LASSO (least absolute shrinkage and selection operator) problem to test different choices of the coefficient matrices, which correspond to different graph structures, in Algorithm~\ref{algo:full}. We also investigate the influence of stepsizes and the relaxation parameter on the performance of the algorithm.

\paragraph{Fused LASSO:} The famous LASSO problem is to find a sparse least squares solution of a linear system via $\ell_1$-penalization. The fused LASSO \cite{TSRZK05} extends the classical LASSO in the sense that it encourages the sparsity of not only the solution but also the differences between adjacent components of the solution, i.e., the flatness of the solution. Specifically, given $\mu,\nu\in(0,+\infty)$, the fused LASSO can be written as
\begin{align}
\min_{x\in\mathbb{R}^d} \left(\frac{1}{2}\|\mathcal{A} x-b\|^2 + \mu\|x\|_1 + \nu\|Lx\|_1\right),
\end{align}
where $\mathcal{A}\in\mathbb{R}^{m\times d}, b\in\mathbb{R}^m$, and $L\in\mathbb{R}^{(d-1)\times d}$ with $L_{ii}=-1$, $L_{i,i+1}=1$, and $0$ otherwise.

\paragraph{Decentralized fused LASSO:}
We are interested in the decentralized fused LASSO problem
\begin{align}
\label{prob:dec_fused_LASSO}
    \min_{x\in\mathbb{R}^d} \left(\sum_{i=1}^n \frac{1}{2}\|\mathcal{A}_{(i)} x - b_{(i)}\|^2 + \sum_{i=1}^n \mu_{(i)}\|x\|_1 + \sum_{i=1}^n \nu_{(i)}\|Lx\|_1\right),
\end{align}
where $\mathcal{A}_{(i)} \in\mathbb{R}^{m_i\times d}, b_{(i)}\in\mathbb{R}^{m_i}$ for $i\in\{1,\dots,n\}$, $\sum_{i=1}^n m_i = m$, $\sum_{i=1}^n \mu_{(i)} = \mu$, $\sum_{i=1}^n \nu_{(i)} = \nu$, and the same $L \in \mathbb{R}^{(d-1)\times d}$ as above. Problem \eqref{prob:dec_fused_LASSO} can be rewritten in an equivalent form
\begin{align}
\label{prob:re_dec_fused_LASSO}
    \text{find } x\in \mathbb{R}^d \text{ such that } 0 \in \sum_{i=1}^n A_i x + \sum_{i=1}^n L^* B_i L x + \sum_{i=1}^n C_i x,
\end{align}
where $A_i x = \partial (\mu_{(i)}\|x\|_1)$, $L^* B_i Lx = \partial(\nu_{(i)}\|Lx\|_1)$, and $C_i x = \nabla(\frac{1}{2}\|\mathcal{A}_{(i)} x-b_{(i)}\|^2) = \mathcal{A}_{(i)}^\top(\mathcal{A}_{(i)} x - b_{(i)})$. It is easy to see that $C_i$ is $\frac{1}{\ell_i}$-cocoercive with $\ell_i=\|\mathcal{A}_{(i)}^\top\mathcal{A}_{(i)}\|_2=\|\mathcal{A}_{(i)}\|_2^2$. In addition, it follows from \cite{NPR13} that $\|L\|=\sqrt{\lambda_{\max}(L^*L)}$ can be analytically computed as $\sqrt{2-2\cos{\left(\frac{d-1}{d}\pi\right)}}$. 

We address problem \eqref{prob:re_dec_fused_LASSO} by using the technique discussed in Remark~\ref{r:n+1} for three algorithms: complete graph algorithm \eqref{algo:complete}, sequential graph algorithm \eqref{algo:seq}, star graph algorithm \eqref{algo:star}, and choose $\kappa=0$ for simplicity. 

\paragraph{Decentralized fused LASSO for cancer detection}
We apply the decentralized fused LASSO to the problem of ``hot-spot detection'' which was studied in \cite{TW08}. In particular, the goal is to detect the regions of gain or loss in comparative genomic hybridization (CGH) data. CGH is a technique used to measure the DNA copy numbers of selected genes across the genome. In cancer cells, mutations can lead to deletions or amplifications of genes, resulting in fewer or more DNA copies. CGH array experiments report the $\log_2$ ratio of DNA copy numbers in tumor cells relative to those in reference cells. In CGH analysis, positive values indicate potential DNA copy number gains, while negative values suggest possible losses. Consequently, values close to zero correspond to normal genes, whereas profiles exhibiting multiple regions of gains and losses may indicate the presence of cancer. In contrast to the classical fused LASSO, the decentralized approach has an important advantage of preserving patient privacy, which is critical in healthcare data management. 

The results of CGH experiments are often interpreted manually by biologists, which can be time-consuming and may lack accuracy. Therefore, the fused LASSO can be applied for automatic interpretation. In this experiment, we use the data from \cite{TW08} to represent CGH measurements from 2 glioblastoma multiforme (GBM) tumors, which is a very aggressive and common type of primary brain cancer.

\paragraph{Parameter settings and results:}
We consider problem \eqref{prob:dec_fused_LASSO} with $n=10$ and the data are partitioned into row blocks randomly. The dataset $x\in\mathbb{R}^{990}$ is given and the matrix $\mathcal{A}\in\mathbb{R}^{990 \times 990}$ is the identity matrix, following \cite[Section~2]{TW08}. Then $b$ is obtained by adding an independent and identically distributed Gaussian noise with variance $10^{-3}$ to $\mathcal{A}x$. 

For the parameters, we set $\mu = 0.01$, $\nu=5$, and $\mu_{(i)}=\frac{\mu}{n}$, $\nu_{(i)}=\frac{\nu}{n}$ for $i\in\{1,\dots,n\}$. For each algorithm, we compute the corresponding upper bound $\gamma_{\max}$ and $\eta_{\max}$ (or $\eta_{k,\max}$ for the sequential and star graph algorithms) using \eqref{stepsize_complete} and \eqref{stepsize_tree}. We then set $\gamma := \hat{\gamma} \gamma_{\max}$ and $\eta = \hat{\eta}\eta_{\max}$ (or $\eta_k = \hat{\eta}\eta_{k,\max}$), where $\hat{\gamma}, \hat{\eta} \in (0,1)$. Finally, we select $\alpha, \hat{\lambda}\in(0,1)$ and set $\lambda_t := \lambda = \hat{\lambda}(1-\alpha)$.

We measure the performance of the algorithms using the relative error 
\begin{align}
    \max_{i\in\{1,\dots,n+1\}} \frac{\|x_i^k - x^*\|}{\|x^*\|},
\end{align}
where the exact solution $x^*$ is computed by CLARABEL v0.11.1 solver, called via CVXPY v1.8.1. The settings for the solver, including absolute duality gap tolerance, relative duality gap tolerance, feasibility check tolerance are adjusted to $10^{-14}$ to achieve higher precision. For simplicity, we examine the effect of varying parameters $\hat{\gamma}, \hat{\eta}, \hat{\lambda}$, and $\alpha$. The behaviors of the sequential and star graph algorithms are very similar, thus we only include the results for the sequential graph algorithm.

First, we examine the effect of $\gamma$ on the performance by fixing $\hat{\eta} = 0.9$, $\hat{\lambda}=0.9$, $\alpha=0.1$, and varying $\hat{\gamma}\in(0,1)$. The results for the complete graph and the sequential graph algorithms are provided in Figure~\ref{fig2:num_2}. They indicate that both algorithms perform better when $\hat{\gamma}$ is small. Repeating the experiment with varying $\hat{\eta}$, we observe the opposite behavior: the algorithms achieve the best performance when $\hat{\eta}$ is large.
\begin{figure}
    \centering
    \begin{subfigure}{0.48\textwidth}
        \includegraphics[width=\linewidth]{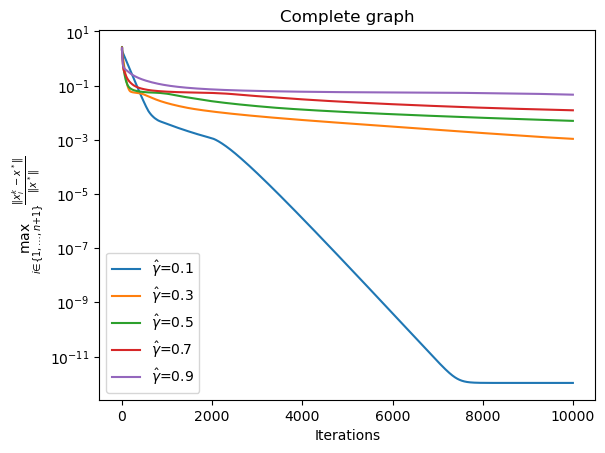}
    \end{subfigure}
    \begin{subfigure}{0.48\textwidth}
        \includegraphics[width=\linewidth]{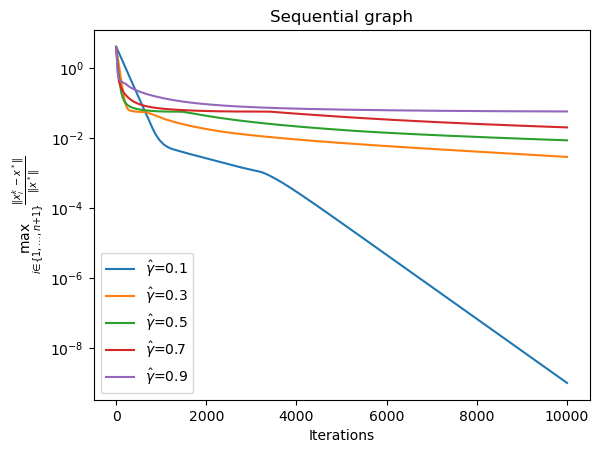}
    \end{subfigure}
    \caption{Effect of varying $\hat{\gamma}$ on the performance of the complete and sequential graph algorithms.}
    \label{fig2:num_2}
\end{figure}
Next, we investigate how the algorithms behave as $\hat{\lambda}$ varies while $\hat{\gamma} = 0.1$, $\hat{\eta}=0.9$, and $\alpha=0.1$ are fixed. Figure~\ref{fig4:num_4} presents the results for the complete graph and the sequential graph algorithms, showing that both algorithms converge faster when $\hat{\lambda}$ is large. In contrast, when varying $\alpha$ within $(0,1)$, we observe that the algorithms perform best for small values of $\alpha$. Similar behaviors with respect to $\hat{\gamma}$ and $\hat{\lambda}$ were also reported in \cite[Section~5.2]{ABT23}. These behaviors may be explained by a trade-off between aggressiveness and stability: while larger values of $\hat{\gamma}$ lead to more aggressive updates, smaller values can produce more stable and better-conditioned iterations, resulting in faster practical convergence. As also noted in \cite[Remark~2.8]{MT20}, the optimal convergence rate does not always occur for the largest possible stepsize.
\begin{figure}
    \centering
    \begin{subfigure}{0.48\textwidth}
        \includegraphics[width=\linewidth]{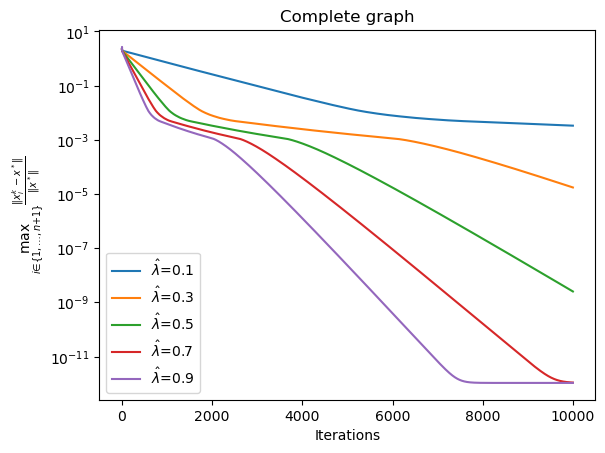}
    \end{subfigure}
    \begin{subfigure}{0.48\textwidth}
        \includegraphics[width=\linewidth]{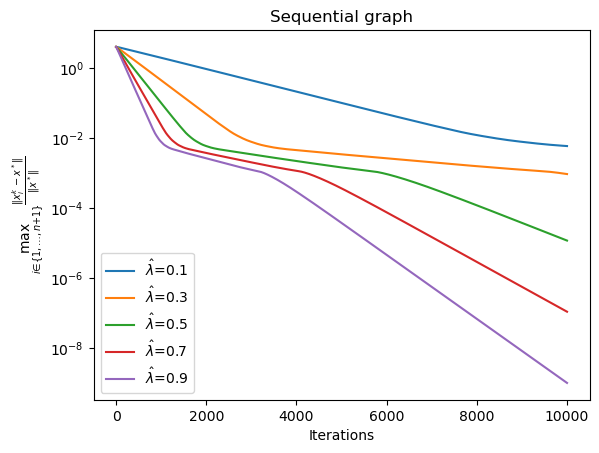}
    \end{subfigure}
    \caption{Effect of varying $\hat{\lambda}$ on the performance of the complete and sequential graph algorithms.}
    \label{fig4:num_4}
\end{figure}

Now, we compare the performance of three algorithms using $\hat{\gamma}=0.1, \hat{\eta}=0.9$, $\hat{\lambda}=0.9$, and $\alpha=0.1$. In Figure~\ref{fig5:sub1}, the data are represented by dots, while the solution obtained by the CVXPY solver and the complete graph algorithm are shown by red solid and green dashed lines, respectively. The result suggests that the complete graph solution successfully detects both regions of gains and losses and matches the solution given by the CVXPY solver mentioned above. The relative error shown in Figure~\ref{fig5:sub2} indicates that the complete graph converges faster in terms of iteration counts compared to the others. However, this comes at the cost of increased runtime. The sequential graph and star graph algorithms exhibit slightly different at the beginning but nearly identical after that.
\begin{figure}
\captionsetup{justification=centering}
    \centering
    \begin{subfigure}{0.49\textwidth}
        \includegraphics[width=\linewidth]{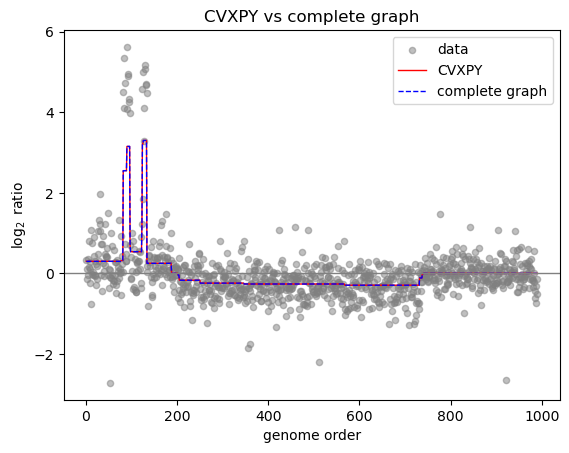}
        \caption{CGH data with solutions obtained by the CVXPY solver and the complete graph algorithm.}
        \label{fig5:sub1}
    \end{subfigure}\hfill
    \begin{subfigure}{0.49\textwidth}
        \includegraphics[width=\linewidth]{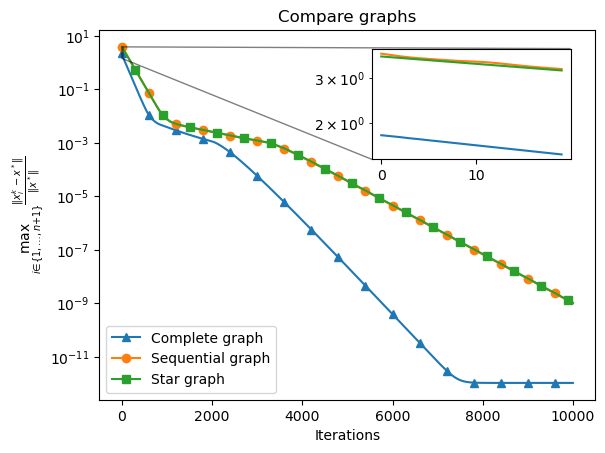}
        \caption{Convergence comparison of the complete graph, sequential graph, and star graph algorithms.}
        \label{fig5:sub2}
    \end{subfigure}
    \caption{Comparison of the solutions obtained by graph-based algorithms and the CVXPY solver.}
    \label{fig5:num_5}
\end{figure}

We note that the performance of the algorithms can indeed vary depending on the coefficient matrices, as reflected in the numerical experiments. Choosing these matrices is, in general, a nontrivial task that has been studied independently \cite{BB26, BB26b}. The primary aim of our numerical experiments is to demonstrate the generality and applicability of the proposed framework. We focus on clarity and accessibility, providing interpretations for specific scenarios rather than performing an exhaustive sensitivity analysis or drawing broad conclusions regarding performance.

\paragraph{Acknowledgements.} The research of MND, MKT and TDT was supported in part by Australian Research Council grant DP230101749.

\end{document}